\documentclass[final,leqno,showlabe]{siamltex}
\usepackage{amsmath}
\usepackage{stmaryrd}
\usepackage{booktabs}
\usepackage{graphicx}
\usepackage{color}
\usepackage{cite}
\usepackage{amssymb}
\usepackage{verbatim}
\usepackage[all,cmtip]{xy}
\usepackage[top=3.54cm,bottom=4.54cm,left=2.94cm,right=2.94cm ]{geometry}

\input psfig.sty

\renewcommand{\theequation}{\arabic{section}.\arabic{equation}}
\renewcommand{\thetable}{\arabic{section}.\arabic{table}}
\renewcommand{\thefigure}{\arabic{section}.\arabic{figure}}
\newcommand{\bx}{{\bf x} }

\newcommand{\p}{\partial}
\newcommand{\eps}{\varepsilon}

\newtheorem{rmk}{Remark}[section]
\newcommand{\be}{\begin{equation}}
\newcommand{\ee}{\end{equation}}
\newcommand{\ba}{\begin{array}}
\newcommand{\ea}{\end{array}}
\newcommand{\bea}{\begin{eqnarray}}
\newcommand{\eea}{\end{eqnarray}}
\newcommand{\beas}{\begin{eqnarray*}}
\newcommand{\eeas}{\end{eqnarray*}}

\title{Uniform error bounds of time-splitting methods for the nonlinear Dirac equation in the nonrelativistic regime without magnetic potential\thanks{This work was partially supported by the Ministry of Education of Singapore grant R-146-000-290-114 (W. Bao and J. Yin) and  NSFC grants 11771036 and 91630204 (Y. Cai).}}
\author{Weizhu Bao\thanks{Department of Mathematics, National University of Singapore, Singapore
		119076 ({\tt matbaowz@nus.edu.sg},
		URL: http://blog.nus.edu.sg/matbwz/)}
	\and Yongyong Cai\thanks{Corresponding author. Laboratory of Mathematics and Complex Systems (Ministry of Education), School of Mathematical Sciences, Beijing Normal University, Beijing 100875, P. R. China, and Beijing Computational Science Research Center,
		Beijing 100193, P. R. China ({\tt yongyong.cai@bnu.edu.cn})}
	\and Jia Yin\thanks{NUS Graduate School for Integrative Sciences and Engineering (NGS),
		National University of Singapore, Singapore 117456 ({\tt yinjia15@u.nus.edu})}}
\begin{document}

\maketitle

\begin{abstract}
Super-resolution of the Lie-Trotter splitting ($S_1$) and Strang splitting ($S_2$)  is rigorously analyzed for the nonlinear Dirac equation without external magnetic potentials in the nonrelativistic regime with a small parameter $0<\eps\leq 1$ inversely proportional to the speed of light.  In this regime, the solution highly oscillates in time with   wavelength at $O(\eps^2)$. The splitting methods surprisingly show super-resolution, i.e. the methods can capture the solution accurately even if the time step size $\tau$ is much larger than the sampled wavelength at $O(\eps^2)$. Similar to the linear case, $S_1$ and $S_2$ both exhibit $1/2$ order convergence uniformly with respect to $\eps$. Moreover, if $\tau$ is non-resonant, i.e. $\tau$ is away from certain region determined by $\eps$, $S_1$ would yield an improved uniform first order $O(\tau)$ error bound, while $S_2$ would give  improved uniform $3/2$ order convergence. Numerical results are reported to confirm these rigorous results. Furthermore, we note that super-resolution is still valid for higher order splitting methods.
\end{abstract}



\section{Introduction}\setcounter{equation}{0}

The splitting methods form an important group of methods which are quite accurate and efficient \cite{MQ}. Actually, they have been widely applied for dealing with highly oscillatory systems such as the Schr\"{o}dinger/nonlinear Schr\"{o}dinger equations \cite{ABB, BJM, BJM2, Carles, CG,Lubich,Tha}, the Dirac/nonlinear Dirac equations \cite{BCJT2, BCJY, BY, LLS}, the Maxwell-Dirac system \cite{BL, HJMSZ}, the Zakharov system \cite{BS2, BSW, Gauckler, JMZ}, the Gross-Pitaevskii equation for Bose-Einstein condensation (BEC) \cite{BS}, the Stokes equation \cite{CHP}, and the Enrenfest dynamics \cite{FJS}, etc.

In this paper, we consider the splitting methods applied to the nonlinear Dirac equation (NLDE) \cite{Dirac, Dirac2,FLR, FS, Heis,CEL, HC, HWC,FLW, FW, FW2, FW3, Sah,ST2006,XST} in the nonrelativistic regime without magnetic potential.  In one or two dimensions (1D or 2D), the equation can be represented in the  two-component form with wave function $\Phi := \Phi(t, \bx) = (\phi_1(t, \bx), \phi_2(t, \bx))^T\in \mathbb{C}^2$ \cite{BCJY}:
\be\label{eq:NLDirac1d2d}
i\partial_t\Phi = \left(-\frac{i}{\eps}\sum_{j=1}^d\sigma_j\partial_j + \frac{1}{\eps^2}\sigma_3\right)\Phi + V(\bx)\Phi + \mathbf{F}(\Phi)\Phi, \quad \bx\in\mathbb{R}^d, \quad d = 1, 2, \quad t>0,
\ee
where $i=\sqrt{-1}$ is the imaginary unit, $t$ is time, $\bx=(x_1, ..., x_d)^T$, $\partial_j=\frac{\partial}{\partial x_j}$ ($j=1,...,d$), $\eps\in(0, 1]$ is a dimensionless parameter inversely proportional to the speed of light, and $V:=V(\bx)$ is a real-valued function denoting the external electric potential. $\sigma_1$, $\sigma_2$, $\sigma_3$ are the Pauli matrices defined as
\be\label{Paulim}
\sigma_{1}=\left(
\begin{array}{cc}
	0 & 1  \\
	1 & 0  \\
\end{array}
\right), \qquad
\sigma_{2}=\left(
\begin{array}{cc}
	0 & -i \\
	i & 0 \\
\end{array}
\right),\qquad
\sigma_{3}=\left(
\begin{array}{cc}
	1 & 0 \\
	0 & -1 \\
\end{array}
\right).
\ee
The nonlinearity $\mathbf{F}(\Phi)$ in \eqref{eq:NLDirac1d2d} is usually taken as
\be\label{eq:F}
\mathbf{F}(\Phi) = \lambda_1(\Phi^*\sigma_3\Phi)\sigma_3 + \lambda_2|\Phi|^2I_2,
\ee
with $|\Phi|^2 = \Phi^*\Phi$, where $\lambda_1$, $\lambda_2\in\mathbb{R}$ are two given real constants,  $\Phi^*=\overline{\Phi}^T$ is the complex conjugate transpose of $\Phi$ and
$I_2$ is the $2\times 2$ identity matrix.
The above choice of nonlinearity is motivated from the so-called Soler model in quantum field theory, e.g. $\lambda_2=0$ and $\lambda_1\ne0$ \cite{FLR,FS,Thirring}, and BEC with a chiral confinement and/or spin-orbit coupling, e.g. $\lambda_1=0$ and $\lambda_2\ne0$ \cite{CEL,HC,HWC}.
In order to study the dynamics, the initial data is chosen as
\be\label{eq:initial}
\Phi(t = 0, \bx) = \Phi_0(\bx), \quad \bx \in\mathbb{R}^d, \quad d = 1, 2.
\ee


When  $\eps = 1$ in \eqref{eq:NLDirac1d2d}, which corresponds to the classical regime of the nonlinear Dirac equation, there have been comprehensive analytical and numerical results in the literatures. In the analytical aspect, for the existence and multiplicity of bound states and/or standing wave solutions, we refer to \cite{Bala,Bala1,Bar,Caz,Dol,Esteban0,Esteban,Komech} and references therein. Particularly, for the case where $d=1$, $V(x)\equiv0$,  $\lambda_1=-1$ and $\lambda_2=0$ in the choice of $\mathbf{F}(\Phi)$, the NLDE \eqref{eq:NLDirac1d2d} admits explicit soliton solutions \cite{CKMS,FS,Hag,Kor,Mat2,Raf,Stu,Tak}. In the numerical aspect, many accurate and efficient numerical methods have been proposed and analyzed,  such as the finite difference time domain (FDTD) methods \cite{BHM,Hamm,NSG}, the time-splitting Fourier spectral (TSFP) methods \cite{BL,BZ,FSS,HJMSZ} and the Runge-Kutta discontinuous Galerkin methods \cite{HL}.

On the other hand, when $0<\eps\ll 1$ (the nonrelativistic regime where the wave speed is much smaller than the speed of light),  as indicated by previous analysis in \cite{BCJY, Foldy, N, CW}, the wavelength of the solution in time is at $O(\eps^2)$. The oscillation of the solution as well as the unbounded and indefinite energy functional w.r.t. $\eps$ \cite{BMP, Esteban}  cause much burden in the analysis and computation. Indeed, it would require that the time step size $\tau$ to be strictly reliant on $\eps$ to capture the exact solution. Numerical studies in \cite{BCJY} have confirmed this dependence. The error bounds show that $\tau = O(\eps^3)$ is required for the conservative Crank-Nicolson finite difference (CNFD) method \cite{BCJY}, and $\tau = O(\eps^2)$ is required for the exponential wave integrator Fourier pseudospectral (EWI-FP) method as well as the time-splitting Fourier pseudospectral (TSFP) method \cite{BCJY}. To overcome the restriction, recently, uniform accurate (UA) schemes with two-scale formulation approach \cite{LMZ} or multiscale time integrator pseudospectral method \cite{BCJT,CW2} or nested Picard iterative integrators \cite{CW1} have been designed for the NLDE in the nonrelativistic regime, where the time step size $\tau$ could be independent of $\eps$.

Though the error of the TSFP method (also called $S_2$ later in this paper) has a $\tau^2/\eps^4$ dependence on the small parameter $\eps$ \cite{BCJY}, under the specific case where there is a lack of magnetic potential, as in \eqref{eq:NLDirac1d2d}, we find out through our recent extensive numerical experiments that the error of $S_2$ is independent of $\eps$ and uniform w.r.t. $\eps$. In other words, $S_2$ for the NLDE \eqref{eq:NLDirac1d2d} in the absence of magnetic potentials displays \textbf{super-resolution} w.r.t. $\eps$.


The \textbf{super-resolution} here suggests independence of the oscillation wavelength. It is even stronger than the `super-resolution' in \cite{DT} for the Schr\"{o}dinger equation in the semiclassical regime, where the restriction on the time steps is still related to the wavelength, but not so strict as the resolution of the oscillation by fixed number of points per wavelength. This property for the time-splitting methods makes them superior in solving the NLDE in the absence of magnetic potentials in the nonrelativistic regime as they are more efficient and reliable as well as simple compared to other numerical methods in the literature. In this paper, the super-resolution for the first-order ($S_1$) and second-order ($S_2$) time-splitting methods will be rigorously analyzed, and numerical results will be presented to validate the conclusions. We remark that similar results have been analyzed for the Dirac equation \cite{BCY}, where the linearity enables us to explicitly track the error exactly and make estimation at the target time step without using Gronwall type arguments. However, in the nonlinear case, it is impossible to follow the error propagation exactly and estimations have to be  done at each time step. As a result, Gronwall arguments will be involved together with the mathematical induction  to control the nonlinearity and to bound the numerical solution. In particular,   instead of the previously adopted Lie calculus approach \cite{Lubich}, Taylor expansion  and Duhamel principle are employed to study the local error of the splitting methods, which can identify how temporal oscillations propagate numerically. In other words, the techniques adopted to establish uniform error bounds of the time-splitting methods for the NLDE are completely different with those used for the Dirac equation \cite{BCY}.

The rest of the paper is organized as follows. In section 2, we establish uniform error estimates of the first-order time-splitting method for the NLDE without magnetic potentials in the nonrelativistic regime and
report numerical results to confirm our uniform error bounds.
Similar results are presented for the second-order time-splitting method
in section 3 with a remark on extension to higher order splitting methods.
Some conclusions are drawn in section 4.
Throughout the paper, we adopt the standard Sobolev spaces  and the corresponding norms. Meanwhile, $A \lesssim B$ is used in the sense that there exists a generic constant $C > 0$ independent of $\eps$ and $\tau$,
such that $|A|\le C\,B$. $A \lesssim_\delta B$ has a similar meaning that there exists a generic constant $C_\delta>0$ dependent on $\delta$ but independent of $\eps$ and $\tau$, such that $|A|\le C_\delta\,B$.

\section{Uniform error bounds of the first-order Lie-Trotter splitting method}\setcounter{equation}{0}
For simplicity of notations and without loss of generality, here we only consider \eqref{eq:NLDirac1d2d} in 1D ($d = 1$). Extensions to \eqref{eq:NLDirac1d2d} in 2D and/or  the four component form of
the NLDE 
 with $d = 1, 2, 3$ \cite{BCJY} are straightforward.

Denote the free Dirac Hermitian operator
\be
Q^\eps = -i\eps\sigma_1\partial_x + \sigma_3, \quad x\in\mathbb{R},
\ee
then the NLDE \eqref{eq:NLDirac1d2d} in 1D can be written as
\be\label{eq:NLDirac_nonmag}
i\partial_t\Phi(t, x) = \frac{1}{\eps^2}Q^\eps\Phi(t, x) + V(x)\Phi(t, x) + \mathbf{F}(\Phi(t, x))\Phi(t, x), \quad x\in\mathbb{R},
\ee
with nonlinearity \eqref{eq:F} and the initial condition \eqref{eq:initial}.

Choose $\tau > 0$ as the time step size and $t_{n}=n\tau$ for $n = 0, 1, ...$ as the time steps.
Denote $\Phi^n(x)$ to be the numerical approximation of $\Phi(t_n,x)$,
where $\Phi(t,x)$ is the exact solution of \eqref{eq:NLDirac_nonmag} with \eqref{eq:F} and \eqref{eq:initial}, then through applying the discrete-in-time first-order splitting (Lie-Trotter splitting) \cite{Trotter}, $S_1$ can be represented as \cite{BCJY}:
\be\label{eq:S1}
\Phi^{n + 1}(x) = e^{-\frac{i\tau}{\eps^2}Q^\eps}e^{-i\tau\left[V(x) + \mathbf{F}(\Phi^n(x))\right]}\Phi^n(x), \quad \text{with } \Phi^0(x) = \Phi_0(x), \quad x\in\mathbb{R}.
\ee
For simplicity, we also write $\Phi^{n+1}(x):=S_{n,\tau}^{\text{Lie}}(\Phi^n)$, where $S_{n,\tau}^{\text{Lie}}$ denotes the numerical propagator of the Lie-Trotter splitting.

\subsection{A uniform error bound}
For any $0<T<T^*$, where $T^*$ denotes the common maximal existence time of the solution for \eqref{eq:NLDirac1d2d} with \eqref{eq:F} and \eqref{eq:initial} for all $0<\eps\leq1$, we are going to consider smooth solutions, i.e. we assume the electric potential satisfies
\begin{equation*}
(A)\hskip 5cm  V(x)\in 
W^{2m+1, \infty}(\Bbb R), \; m\in\mathbb{N}^*.\hskip 8cm
\end{equation*}
In addition, we assume the exact solution $\Phi(t, x)$ satisfies
\begin{equation*}
(B)\hskip 4cm \Phi(t, x)\in L^\infty([0, T];(H^{2m+1}(\mathbb{R}))^2), \quad m\in\mathbb{N}^*. \hskip 6cm
\end{equation*}

For the numerical approximation $\Phi^n(x)$ obtained from $S_1$ \eqref{eq:S1}, we introduce the error function
\be\label{eq:error}
{\bf e}^n(x) = \Phi(t_n, x) - \Phi^n(x), \quad 0\leq n\leq \frac{T}{\tau},
\ee
then the following uniform error bound in $H^1$ norm can be established, where the $H^1$ norm for  function $\Phi(x)=(\phi_1,\phi_2)^T\in\mathbb{C}^2$ is given by
\begin{equation}
\|\Phi\|_{H^1}^2=\|\Phi\|_{L^2}^2+\|\partial_x\Phi\|_{L^2}^2,\end{equation}
with $L^2$ norm defined as $\|\Phi\|_{L^2}=\sqrt{\int_{\mathbb{R}}|\Phi(x)|^2\,dx}=\sqrt{\int_{\mathbb{R}}\left(|\phi_1(x)|^2+|\phi_2(x)|^2\right)\,dx}$.

\begin{theorem}\label{thm:lie}
	Let $\Phi^n(x)$ be the numerical approximation obtained from $S_1$ \eqref{eq:S1},
	then under assumptions $(A)$ and $(B)$ with $m = 1$, there exists $0<\tau_0\leq1$ independent of $\eps$ such that
	the following two error estimates hold for  $0<\tau<\tau_0$
	\be\label{eq:thm1}
	\|{\bf e}^n(x)\|_{H^1}\lesssim \tau+\eps,\quad \|{\bf e}^n(x)\|_{H^1}\lesssim \tau+\tau/\eps, \quad 0\le n\le\frac{T}{\tau}.
	\ee
	Consequently, there is a uniform error bound for $S_1$ when $0<\tau<\tau_0$
	\be\label{eq:thm2}
	\|{\bf e}^n(x)\|_{H^1}\lesssim \tau + \max_{0<\eps\leq 1}\min\{\eps, \tau/\eps\} \lesssim \sqrt{\tau}, \quad 0\le n\le\frac{T}{\tau}.
	\ee
\end{theorem}

\begin{rmk}
Instead of proving the $L^2$ error bounds as in the linear case, in Theorem \ref{thm:lie} and the other results in this paper for the 1D problem, we prove the $H^1$ error bounds for ${\bf e}^n(x)$ due to the fact that $H^1(\mathbb{R})$ is an algebra, and the corresponding estimates should be in $H^2$ norm for 2D and 3D cases (2D case in the sense of \eqref{eq:NLDirac1d2d}, and 3D case in the sense of the four-component nonlinear Dirac equation given in \cite{BCJY}) with of course higher regularity assumptions (higher order Sobolev norm estimates need higher regularity of the exact solution). 
\end{rmk}

\begin{rmk}
	In Theorem \ref{thm:lie}, the $H^3$ regularity ($m=1$ in assumptions (A) and (B)) is assumed for the first order Lie splitting scheme, and this regularity assumption is sharp for the results stated in the theorem. Heuristically, the estimates of the type $\|{\bf e}^n(x)\|_{H^1}\lesssim \tau+\eps$ hold for $\varepsilon\in(0,1]$, while in the limit $\varepsilon\to0^+$, the NLDE \eqref{eq:NLDirac1d2d} converges to the coupled nonlinear Schr\"odinger equations (CNLSE) after filtering out the nonrelativistic temporal oscillations \cite{BCJY, Foldy, N, CW}.  Thus, letting $\varepsilon\to0^+$,  the estimates $\|{\bf e}^n(x)\|_{H^1}\lesssim \tau+\eps$ will become the error bounds for the Lie splitting method applied to CNLSE. For the $H^1$ error estimates of the  Lie splitting  in the case of Schr\"odinger type equations, the regularity requirement of the exact solution should be 2 orders higher \cite{Lubich} (one oder temporal derivative corresponds to two order spatial derivative in Schr\"odinger type equations), i.e. $H^3$ regularity of the exact solution is needed.
\end{rmk}

\smallskip

{\sl For simplicity of the presentation, in the proof for this theorem and other theorems later for NLDE in this paper, we take $V(x) \equiv 0$}. Extension to the case where
$V(x)\ne 0$ is straightforward \cite{BCY}.  Compared to the linear case \cite{BCY}, the nonlinear term is much more complicated to analyze.  As discussed earlier in the introduction, for the linear Dirac equation, the linearity and $L^2$ unitary property of the numerical propagator enable the explicit expression (exact) of the error ${\bf e}^n(x)$ by the local error (see Lemma \ref{lemma:lie}) without any extra condition on time step $\tau$. Therefore, the error estimates (in $L^2$ norm)  in \cite{BCY} are obtained by carefully studying the accumulation of the local errors. However, for the nonlinear Dirac equation case, the approach (highly depend on the linear property) in \cite{BCY} fails. Different from the linear case, the novelty of the strategy we adopt for the nonlinear case  lies in the following aspects: (i)  carefully carry out   expansions of the nonlinear terms to analyze the local errors and identify the leading temporal oscillations; and (ii) estimate the errors in $H^1$ norm where the conditional stabilities of the numerical propagators (see Lemma \ref{lem:lie:stab}) and  the NLDE \eqref{eq:NLDirac1d2d} hold, and then control the nonlinear terms by mathematical induction, the uniform error estimates \eqref{eq:thm2} and Sobolev inequalities (see \eqref{eq:indu} and the proof after).  We emphasis here that our analysis and convergence rate results are valid for the linear Dirac equation, while the approach and error estimates in the linear case \cite{BCY} can not be applied here for the nonlinear case. Of course, the dependence of the constant on time $T$ in front of the convergence rate is sharper in the linear case in \cite{BCY} than that in Theorem 2.1.

    As mentioned above, a key issue of the error analysis for NLDE is to control the nonlinear term of numerical solution $\Phi^n$, and for which we require the following stability lemma \cite{Lubich}.
\begin{lemma}\label{lem:lie:stab} Suppose $V(x)\in W^{1,\infty}(\mathbb{R})$, and $\Phi(x),\Psi(x)\in (H^1(\mathbb{R}))^2$ satisfy
$\|\Phi\|_{H^1},\|\Psi\|_{H^1} \leq M$, we have
\begin{equation}
\|S_{n,\tau}^{\rm{Lie}}(\Phi)-S_{n,\tau}^{\rm{Lie}}(\Psi)\|_{H^1}\leq e^{c_1\tau}\|\Phi-\Psi\|_{H^1},
\end{equation}
where $c_1$ depends on $M$ and $\|V(x)\|_{W^{1,\infty}}$.
\end{lemma}
\begin{proof} The proof is quite similar to the nonlinear Schr\"odinger equation case in \cite{Lubich} and we omit it here for brevity.
\end{proof}

Under the assumption (B) ($m\ge1$), for $\eps\in(0,1]$, we denote $M_1>0$ as
\be\label{eq:M1}
M_1=\sup\limits_{\eps\in(0,1]}\|\Phi(t,x)\|_{L^\infty([0,T]; (H^1(\mathbb{R}))^2)}.
\ee
Based on \eqref{eq:M1} and Lemma \ref{lem:lie:stab}, one can control the nonlinear term once the hypothesis of the lemma is fulfilled. Making use of the fact that $S_1$ is explicit, together with the uniform error estimates in Theorem \ref{thm:lie}, we can use mathematical induction to complete the proof.

The following properties of $Q^\eps$ will be frequently used in the analysis. $Q^\eps$ is diagonalizable in the phase space (Fourier domain) and can be decomposed as
\be\label{eq:Teps}
Q^\eps = \sqrt{Id - \eps^2\Delta} \;\Pi_+^\eps - \sqrt{Id - \eps^2\Delta}\; \Pi_-^\eps,
\ee
where $\Delta = \partial_{xx}$ is the Laplace operator in 1D, $Id$ is the identity operator, and $\Pi_+^\eps$, $\Pi_-^\eps$ are projectors defined as
\be
\Pi_+^\eps = \frac{1}{2}\left[Id + (Id - \eps^2\Delta)^{-1/2}Q^\eps\right], \quad \Pi_-^\eps = \frac{1}{2}\left[Id - (Id - \eps^2\Delta)^{-1/2}Q^\eps\right].
\ee
It is straightforward to verify that
$\Pi_+^\eps + \Pi_-^\eps = Id$, $\quad \Pi_+^\eps\Pi_-^\eps = \Pi_-^\eps\Pi_+^\eps = 0$, $\quad (\Pi_\pm^\eps)^2 = \Pi_\pm^\eps$, and through Taylor expansion, we have \cite{BMP}
\begin{align}\label{eq:pi+}
\Pi_\pm^\eps=\Pi_\pm^0\pm\eps\mathcal{R}_1=\Pi_\pm^0\mp i\frac{\eps}{2}\sigma_1\partial_x\pm\eps^2\mathcal{R}_2,\quad \Pi_+^0=\text{diag}(1,0),\quad \Pi_-^0=\text{diag}(0,1),
\end{align}
with $\mathcal{R}_1: (H^m(\mathbb{R}))^2 \rightarrow (H^{m-1}(\mathbb{R}))^2$ for $m\geq 1$, $\mathcal{R}_2: (H^m(\mathbb{R}))^2\rightarrow(H^{m-2}(R))^2$ for $m\geq 2$ being uniformly bounded operators w.r.t. $\eps$. For simplicity of expression, we denote
\begin{equation}
\Phi_\pm^\eps(t, x):=\Pi_\pm^\eps\Phi(t, x).
\end{equation}
In order to characterize the oscillatory features of the solution, noticing $(\sqrt{Id - \eps^2\Delta} - Id)(\sqrt{Id - \eps^2\Delta} + Id)=-\eps^2\Delta$, we denote
\be\label{eq:Deps}
\mathcal{D}^\eps = \frac{1}{\eps^2}(\sqrt{Id - \eps^2\Delta} - Id) = -(\sqrt{Id - \eps^2\Delta} + Id)^{-1}{\Delta},
\ee
which is a uniformly bounded operator w.r.t $\eps$ from $(H^m(\mathbb{R}))^2\rightarrow(H^{m-2}(\mathbb{R}))^2$ for $m\geq 2$, then the evolution operator $e^{\frac{it}{\eps^2}Q^\eps}$ can be expressed as
\be
\begin{aligned}\label{eq:dec:ev}
	e^{\frac{it}{\eps^2}Q^\eps} =  e^{\frac{it}{\eps^2}(\sqrt{Id - \eps^2\Delta}\Pi_+^\eps - \sqrt{Id - \eps^2\Delta}\Pi_-^\eps)} =e^{\frac{it}{\eps^2}}e^{it\mathcal{D}^\eps}\Pi_+^\eps + e^{-\frac{it}{\eps^2}}e^{-it\mathcal{D}^\eps}\Pi_-^\eps.
\end{aligned}
\ee
For simplicity, here we use $\Phi(t) := \Phi(t, x)$, $\Phi^n := \Phi^n(x)$ in short.

%
Now we are ready to introduce the following lemma for proving Theorem \ref{thm:lie}.

\begin{lemma}\label{lemma:lie}
	Let $\Phi^n(x)$ ( $0\leq n\leq \frac{T}{\tau} - 1$)  be obtained from $S_1$ \eqref{eq:S1} satisfying $\|\Phi^n(x)\|_{H^1}\leq M_1+1$,
 under the assumptions of Theorem \ref{thm:lie}, we have
	\be
	{\bf e}^{n+1}(x) = e^{-\frac{i\tau}{\eps^2}Q^\eps}e^{-i\tau\mathbf{F}(\Phi^n)}{\bf e}^n(x) + \eta_1^n(x) + e^{-\frac{i\tau}{\eps^2}Q^\eps}\eta_2^n(x),
	\ee
	with $\|\eta_1^n(x)\|_{H^1}\leq c_1\tau^2 + c_2\tau\|{\bf e}^n(x)\|_{H^1}$, $\eta_2^n(x) = \int_0^\tau f_2^n(s)ds - \tau f_2^n(0)$, where $c_1$ depends on $M_1$, $\lambda_1$,  $\lambda_2$ and $\|\Phi(t,x)\|_{L^\infty([0,T];(H^3)^2)}$; $c_2$ depends on $M_1$, $\lambda_1$, and $\lambda_2$. Here
	\begin{align}
	f_2^n(s) =& -ie^{\frac{-4is}{\eps^2}}\Pi_-^\eps\left(\mathbf{g}_1^n(x)\Phi_+^\eps(t_n)\right) -i e^{\frac{4is}{\eps^2}}\Pi_+^\eps\left(\overline{\mathbf{g}_1^n}(x)\Phi_-^\eps(t_n)\right)\nonumber\\
	& -i e^{\frac{-i2s}{\eps^2}}\left[\Pi_+^\eps\left(\mathbf{g}_1^n(x)\Phi_+^\eps(t_n)\right) + \Pi_-^\eps\left(\mathbf{g}_2^n(x)\Phi_+^\eps(t_n)+\mathbf{g}_1^n(x)\Phi_-^\eps(t_n)  \right)\right]\nonumber\\
	& - ie^{\frac{2is}{\eps^2}}\left[\Pi_-^\eps\left(\overline{\mathbf{g}_1^n}(x)\Phi_-^\eps(t_n)\right) + \Pi_+^\eps\left(\mathbf{g}_2^n(x)\Phi_-^\eps(t_n) + \overline{\mathbf{g}_1^n}(x)\Phi_+^\eps(t_n)\right)\right] \label{eq:f2nlie},
	\end{align}
where $\mathbf{g}_{j}^n(x)=\mathbf{g}_j(\Phi_+^\eps(t_n),\Phi_-^\eps(t_n))$ and
	\begin{align}
	& \mathbf{g}_1(\Phi_+^\eps(t_n),\Phi_-^\eps(t_n)) = \lambda_1\left((\Phi_-^\eps(t_n))^*\sigma_3\Phi_+(t_n)\right)\sigma_3+
	\lambda_2\left((\Phi_-^\eps(t_n))^*\Phi_+^\eps(t_n)\right)I_2, \label{eq:g1}\\
	& \mathbf{g}_2(\Phi_+^\eps(t_n),\Phi_-^\eps(t_n)) = \sum_{\sigma=\pm}\left[\lambda_1((\Phi_\sigma^\eps(t_n))^*\sigma_3
\Phi_\sigma^\eps(t_n))\sigma_3+\lambda_2|\Phi_\sigma^\eps(t_n)|^2I_2\right].
	\label{eq:g2}
	\end{align}
\end{lemma}
\begin{proof}
	Through the definition of ${\bf e}^n(x)$ \eqref{eq:error}, noticing the  formula \eqref{eq:S1}, we have
	\be\label{eq:S1_error}
	\mathbf{e}^{n+1}(x) = e^{-\frac{i\tau}{\eps^2}Q^\eps}e^{-i\tau\mathbf{F}(\Phi^n)}\mathbf{e}^n(x) + \eta^n(x), \quad 0\leq n\leq\frac{T}{\tau} - 1, \quad x\in\mathbb{R},
	\ee
	where $\eta^n(x)$ is the ``local truncation error" (notice that this is not the usual local truncation error, compared with $ \Phi(t_{n + 1}, x) - S_{n,\tau}^{\text{Lie}}\Phi(t_n, x)$),
	\be\label{eq:S1_lt}
	\eta^n(x) = \Phi(t_{n + 1}, x) - e^{-\frac{i\tau}{\eps^2}Q^\eps}e^{-i\tau\mathbf{F}(\Phi^n)}\Phi(t_n, x), \quad x\in\mathbb{R}.
	\ee
	By Duhamel's principle, the solution $\Phi(t,x)$ to \eqref{eq:NLDirac_nonmag}  satisfies
	\be\label{eq:Du2}
	\Phi(t_n+s, x) = e^{-\frac{is}{\eps^2}Q^\eps}\Phi(t_n, x) - i\int_0^s e^{-\frac{i(s-w)}{\eps^2}Q^\eps}\mathbf{F}(\Phi(t_n+w, x))\Phi(t_n+w, x)dw,\quad 0\leq s\leq \tau,
	\ee
	which implies that $\|\Phi(t_n+s,x)-e^{-\frac{is}{\eps^2}Q^\eps}\Phi(t_n, x)\|_{H^1}
	\lesssim \tau$ ($s\in[0,\tau]$).
Setting $s=\tau$ in \eqref{eq:Du2}, we have from \eqref{eq:S1_lt},
	\be\label{eq:S1_eta}
		 \eta^n(x) =e^{-\frac{i\tau}{\eps^2}Q^\eps} \left(\int_0^\tau f^n(s)ds-\tau f^n(0)\right)+ R_1^n(x) + R_2^n(x) ,
	\ee
	where
	\begin{align}\label{eq:fns}
	&f^n(s)=-ie^{\frac{is}{\eps^2}Q^\eps}\left(\mathbf{F}(e^{-\frac{is}{\eps^2}Q^\eps}\Phi(t_n))e^{-\frac{is}{\eps^2}Q^\eps}\Phi(t_n,x)\right),\quad {R}_1^n(x) = e^{-\frac{i\tau}{\eps^2}Q^\eps}\left(\Lambda_1^n(x)
+\Lambda_2^n(x)\right),\\
	&{R}_2^n(x)
	=-i\int_0^{\tau}e^{-\frac{i(\tau - s)}{\eps^2}Q^\eps}\left[\mathbf{F}(\Phi(t_n+s))\Phi(t_n+s) - \mathbf{F}(e^{-\frac{is}{\eps^2}Q^\eps}\Phi(t_n))e^{-\frac{is}{\eps^2}Q^\eps}\Phi(t_n)\right]ds,
	\end{align}
	with
	\begin{align}
	\label{eq:Taylor}
	\Lambda_1^n(x)=-\left(e^{-i\tau\mathbf{F}(\Phi^n)} - \left(I_2 - i\tau\mathbf{F}(\Phi^n)\right)\right)\Phi(t_n),\quad \Lambda_2^n(x)=\left(i\tau\left(\mathbf{F}(\Phi^n) - \mathbf{F}(\Phi(t_n)) \right)\right)\Phi(t_n).
	\end{align}
	Noticing \eqref{eq:M1}, \eqref{eq:Du2},   and the fact that $e^{-isQ^\eps/\eps^2}$ preserves $H^k$ norm, it is not difficult to find
	\be\label{eq:S1_R2}
	\|R_2^n(x)\|_{H^1}\lesssim M_1^2\int_0^\tau\|\Phi(t_n+s,x)-e^{-\frac{is}{\eps^2}Q^\eps}\Phi(t_n, x)\|_{H^1}\,ds\lesssim \tau^2.
	\ee
	On the other hand, from the definition of $\textbf{F}$ and the fact that $H^1(\mathbb{R})$ is an algebra, we have for any $\Phi_j(x) = (\phi_{j1}(x), \phi_{j2}(x))^T\in\mathbb{C}^2$, $j = 1, 2, 3$
	\begin{align}
	    \|(\mathbf{F}(\Phi_2) - \mathbf{F}(\Phi_1))\Phi_3\|_{H^1} &= \|\lambda_1\left[(|\phi_{21}|^2 - |\phi_{11}|^2) - (|\phi_{22}|^2 - |\phi_{12}|^2)\right]\sigma_3\Phi_3 + \lambda_2(|\Phi_2|^2 - |\Phi_1|^2)\Phi_3\|_{H^1}\nonumber\\
	    &\lesssim (\|\Phi_1\|_{H^1} + \|\Phi_2\|_{H^1})\|\Phi_2 - \Phi_1\|_{H^1}\|\Phi_3\|_{H^1}.
	\end{align} 
Having the above inequality, using the assumption that $\|\Phi^n\|_{H^1}\leq M_1+1$, and the Taylor expansion in $\Lambda_1^n(x)$, we get
	\begin{align}\label{eq:S1_R1}
	\|{R}_1^n(x)\|_{H^1}&\lesssim
	\tau^2 \|\Phi^n\|_{H^1}^2\|\Phi(t_n)\|_{H^1} +\tau M_1(M_1+1)\|\Phi^n-\Phi(t_n)\|_{H^1}
	\lesssim \tau^2 + \tau\|{\bf e}^n(x)\|_{H^1}.
	\end{align}
It remains to estimate the $f^n(s)$ part.
Using the decomposition \eqref{eq:dec:ev} and the Taylor exapnsion $e^{i\tau\mathcal{D}^\eps} = Id + O(\tau\mathcal{D}^\eps)$ (in the sense of phase space), we have $e^{\frac{-isQ^\eps}{\eps^2}}\Phi(t_n)=e^{\frac{-is}{\eps^2}}\Phi_+^\eps(t_n)+e^{\frac{is}{\eps^2}}\Phi_-^\eps(t_n)+O(s)$,
\be\label{eq:fn11}
f^n(s)=-i\sum_{\sigma=\pm}e^{\frac{\sigma is}{\eps^2}}\Pi_\sigma^\eps\left\{\mathbf{F}\left(e^{\frac{-is}{\eps^2}}\Phi_+^\eps(t_n)+e^{\frac{is}{\eps^2}}\Phi_-^\eps(t_n)\right)\,\left(e^{\frac{-is}{\eps^2}}\Phi_+^\eps(t_n)+e^{\frac{is}{\eps^2}}\Phi_-^\eps(t_n)\right)\right\}+f_1^n(s),
\ee
where for $s\in[0,\tau]$,
\be
\|f_1^n(s)\|_{H^1}\lesssim \tau \|\Phi(t_n)\|_{H^3}^3\lesssim \tau.
\ee
Since $\mathbf{F}$ is of polynomial type, by direct computation, we can further simplify \eqref{eq:fn11} to get
	\begin{align}
f^n(s)=f_1^n(s)+f_2^n(s)+\tilde{f}^n(s),\quad 0\leq s\leq\tau,
	\end{align}
	where  $f_2^n(s)$ is given in \eqref{eq:f2nlie} and $\tilde{f}^n(s)$ is independent of $s$ as
	\begin{equation}\label{eq:tfn}
	\tilde{f}^n(s)\equiv -i\left[\Pi_+^\eps\left(\mathbf{g}_2^n(x)\Phi_+^\eps(t_n)+\mathbf{g}_1^n(x)\Phi_-^\eps(t_n)  \right)+ \Pi_-^\eps\left(\mathbf{g}_2^n(x)\Phi_-^\eps(t_n)+\overline{\mathbf{g}_1^n}(x)\Phi_+^\eps(t_n)  \right)\right],
	\end{equation}	with $\mathbf{g}_{1,2}^n$ defined in \eqref{eq:g1}-\eqref{eq:g2}.
	
	Now, it is easy to verify that $\eta^n(x)=\eta_1^n(x)+\eta_2^n(x)$ with $\eta_2^n(x)$  given in Lemma \ref{lemma:lie} by choosing
	\be
	\eta_1^n(x)=e^{-\frac{i\tau}{\eps^2}Q^\eps}\left(\int_0^\tau (f_1^n(s)+\tilde{f}^n(s))ds-\tau (f_1^n(0)+\tilde{f}^n(0))\right)+ R_1^n(x) + R_2^n(x).
	\ee
	Noticing that $\tilde{f}^n(s)$ is independent of $s$ and $\|f_1^n(s)\|_{H^1}\lesssim \tau$,
	 combining  \eqref{eq:S1_R2} and \eqref{eq:S1_R1}, we can get
	\begin{align*}
	\|\eta_1^n(x)\|_{H^1} \leq \sum_{j=1}^2\|{R}_j^n(x)\|_{H^1}  +\left\|\int_0^\tau f_1^n(s)ds-\tau f_1^n(0)\right\|_{H^1} \lesssim \tau\|\mathbf{e}^n(x)\|_{H^1}+\tau^2,
	\end{align*}
	which completes  the proof of Lemma \ref{lemma:lie}.
\end{proof}

Now, we proceed to prove Theorem \ref{thm:lie}.\
\begin{proof}
We will prove by induction that the estimates \eqref{eq:thm1}-\eqref{eq:thm2} hold for all time steps $n\leq\frac{T}{\tau}$ together with
\be\label{eq:indu}
\|\Phi^n\|_{H^1}\leq M_1+1.
\ee
Since initially $\Phi^0=\Phi_0(x)$, $n=0$ case is obvious. Assume \eqref{eq:thm1}-\eqref{eq:thm2} and \eqref{eq:indu} hold true for all $0\leq n\leq p\leq\frac{T}{\tau}-1$, then we are going to prove the case $n=p+1$.
	
	From Lemma \ref{lemma:lie}, we have
	\begin{equation}\label{eq:e1l}
	{\bf e}^{n + 1}(x) = e^{-\frac{i\tau}{\eps^2}Q^\eps}e^{-i\tau\mathbf{F}(\Phi^n)}{\bf e}^n(x) + \eta_1^{n}(x) + e^{-\frac{i\tau}{\eps^2}Q^\eps} \eta_2^n(x), \quad 0\leq n\leq p,
	\end{equation}
	with $\|\eta_1^n(x)\|_{H^1}\lesssim\tau^2 + \tau\|{\bf e}^n(x)\|_{H^1}$, ${\bf e}^0=0$ and $\eta_2^n(x)$ given in Lemma \ref{lemma:lie}.
	
Denote $\mathcal{L}_n= e^{-\frac{i\tau}{\eps^2}Q^\eps}\left(e^{-i\tau \mathbf{F}(\Phi^n)}-I_2\right)$ ($0\leq n\leq p\leq\frac{T}{\tau}-1$), and
 it is straightforward to calculate
 \be\label{eq:Lst}
\| \mathcal{L}_n\Psi(x)\|_{H^1}\leq C_{M_1}\tau\|\Psi\|_{H^1},\quad \forall\Psi\in (H^1(\mathbb{R}))^2,
 \ee
 with $C_{M_1}$ only depending on $M_1$. Thus we can obtain from \eqref{eq:e1l} that for $0\leq n\leq p$,
	\begin{align}
		{\bf e}^{n+1}(x) &=  e^{-\frac{i\tau}{\eps^2}Q^\eps}{\bf e}^n(x) + \eta_1^n(x) +  e^{-\frac{i\tau}{\eps^2}Q^\eps}\eta_2^n(x)+\mathcal{L}_n{\bf e}^n(x) &\nonumber\\
		&=  e^{-\frac{2i\tau}{\eps^2}Q^\eps}{\bf e}^{n - 1}(x) +  e^{-\frac{i\tau}{\eps^2}Q^\eps}\left(\eta_1^{n - 1}(x) +  e^{-\frac{i\tau}{\eps^2}Q^\eps}\eta_2^{n - 1}(x)+ \mathcal{L}_{n-1}{\bf e}^{n-1}\right) \nonumber\\
		&\qquad\qquad+ \left(\eta_1^n(x) + e^{-\frac{i\tau}{\eps^2}Q^\eps} \eta_2^n(x)+ \mathcal{L}_n{\bf e}^n\right)&\nonumber\\
		&= ...&\nonumber\\
		&= e^{-i(n+1)\tau Q^\eps/\eps^2}{\bf e}^0(x) + \sum_{k = 0}^n  e^{-\frac{i(n-k)\tau}{\eps^2}Q^\eps}\left(\eta_1^k(x) +  e^{-\frac{i\tau}{\eps^2}Q^\eps}\eta_2^k(x)+\mathcal{L}_k{\bf e}^k(x)\right).&\label{eq:eta:lie}
	\end{align}
	 Since $\|\eta_1^k(x)\|_{H^1}\lesssim \tau^2 + \tau\|{\bf e}^n(x)\|_{H^1}$, $k = 0, 1, ..., n$, and $e^{-is/\eps^2Q^\eps}$ ($s\in\mathbb{R}$) preserves $H^1$ norm, we have from \eqref{eq:Lst}
	\begin{equation}
	\left\|\sum_{k = 0}^ne^{-\frac{i(n-k)\tau}{\eps^2}Q{^\eps}}\left(\eta_1^k(x)+\mathcal{L}_k{\bf e}^k\right)\right\|_{H^1}\lesssim \sum_{k = 0}^n\tau^2 + \sum_{k=0}^n\tau\|{\bf e}^k(x)\|_{H^1} \lesssim\tau + \tau\sum_{k=0}^n\|{\bf e}^k(x)\|_{H^1},
	\end{equation}
	which leads to
	\begin{equation}\label{eq:uniform:s1}
	\|{\bf e}^{n+1}(x)\|_{H^1} \lesssim \tau + \tau\sum_{k=0}^n\|{\bf e}^k(x)\|_{H^1} + \left\|\sum_{k = 0}^n e^{-\frac{i(n-k+1)\tau}{\eps^2}Q{^\eps}}{\eta}_2^k(x)\right\|_{H^1},\quad n\leq p.
	\end{equation}

	To analyze $\eta_2^n(x) = \int_0^\tau f_2^n(s)ds - \tau f_2^n(0)$, using \eqref{eq:pi+}, we can find $f_2^n(s)=O(\eps)$, e.g.	\begin{align*}
	(\Phi_+^\eps(t_n))^*\sigma_3 (\Phi_-^\eps(t_n))=&-\eps(\Phi_+^\eps(t_n))^*\sigma_3 (\mathcal{R}_1\Phi(t_n))+\eps(\mathcal{R}_1\Phi(t_n))^*\sigma_3(\Phi_-^\eps(t_n)),	\end{align*}
	and the other terms in $f_2^n(s)$ can be estimated similarly. As $\mathcal{R}_1:(H^{m})^2\to (H^{m-1})^2$ is uniformly bounded with respect to $\eps\in(0,1]$, we have (with detailed computations omitted)
	\begin{align}
\|f_2^n(\cdot)\|_{L^\infty([0,\tau];(H^1)^2)}\lesssim \eps\|\Phi(t_n)\|_{H^2}^3\lesssim \eps.
	\end{align}
	Noticing the assumptions of  Theorem \ref{thm:lie}, we obtain from \eqref{eq:f2nlie}
	\be\label{eq:S1_ineq}
	\|f_2^n(\cdot)\|_{L^\infty([0,\tau];(H^1)^2)}\lesssim \eps, \quad \|\partial_s(f_2^n)(\cdot)\|_{L^\infty([0,\tau];(H^1)^2)} \lesssim \eps/\eps^2 = 1/\eps,
	\ee
	which leads to
	\be\label{eq:S1_e1}
	\left\|\int_0^\tau f_2^n(s)\,ds-\tau f_2^n(0)\right\|_{H^1}\lesssim \tau\eps.
	\ee
	On the other hand, using Taylor expansion and the second inequality in \eqref{eq:S1_ineq}, we have
	\be\label{eq:S1_e2}
	\left\|\int_0^\tau f_2^n(s)\,ds-\tau f_2^n(0)\right\|_{H^1}\leq
	\frac{\tau^2}{2}\|\partial_sf_2^n(\cdot)\|_{L^\infty([0,\tau];(H^1)^2)} \lesssim \tau^2/\eps.	\ee
	Combining \eqref{eq:S1_e1} and \eqref{eq:S1_e2}, we arrive at
	\be
	\|\eta_2^n(x)\|_{H^1}\lesssim \min\{\tau\eps, \tau^2/\eps\}.
	\ee
	Then from \eqref{eq:uniform:s1}, we get for $n\leq p$
	\begin{align}
	\|{\bf e}^{n+1}(x)\|_{H^1}	\lesssim & n\tau^2 + n\min\{\tau\eps, \tau^2/\eps\} + \tau\sum_{k=0}^n\|{\bf e}^n(x)\|_{H^1}.
	\end{align}
	Using discrete Gronwall's inequality,  we have
	\be
	\|{\bf e}^{n+1}(x)\|_{H^1} \lesssim \tau + \min\{\tau\eps, \tau^2/\eps\},\quad n\leq p,
	\ee
 which shows that \eqref{eq:thm1}-\eqref{eq:thm2} hold for $n=p+1$. It can be checked that all the constants appearing in the estimates  depend only on $M_1,\lambda_1,\lambda_2, T$ and $\|\Phi(t,x)\|_{L^\infty([0,T];(H^3)^2)}$, and
 \be
 \|\Phi^{p+1}\|_{H^1}\leq \|\Phi(t_{p+1})\|_{H^1}+\|{\bf e}^{p+1}\|_{H^1}\leq M_1+C\sqrt{\tau}
 \ee
 for some $C=C(M_1,\lambda_1,\lambda_2, T,\|\Phi(t,x)\|_{L^\infty([0,T];(H^3)^2)})$. Choosing $\tau\leq\frac{1}{C^2}$ will justify \eqref{eq:indu} at $n=p+1$, which finishes the induction process, and the proof for Theorem \ref{thm:lie} is completed.
\end{proof}


\subsection{An improved error bound for non-resonant time steps}
The leading term in the NLDE \eqref{eq:NLDirac_nonmag} 
is $\frac{1}{\eps^2}\sigma_3\Phi$, suggesting that the solution behaves almost periodically in time with periods $2k\pi\eps^2$ ($k\in\mathbb{N}^*$, the periods of $e^{-i\sigma_3/\eps^2}$).
From numerical results, we observe that $S_1$ behave much better than the results in Theorem \ref{thm:lie} when $4\tau$ (which is derived from the proof) is not close to the leading temporal oscillation periods $2k\pi\eps^2$.  In fact, for given $0<\delta\leq1$, define
\be \label{tau_range}
\mathcal{A}_\delta(\eps) := \bigcup_{k = 0}^\infty\left[0.5\eps^2k\pi + 0.5\eps^2\arcsin\delta, 0.5\eps^2(k+1)\pi -0.5 \eps^2\arcsin\delta\right], \quad 0 < \eps \leq 1,
\ee
then when $\tau\in \mathcal{A}_\delta(\eps)$, i.e., when non-resonant time step sizes are chosen, the errors of $S_1$ can be improved.
To illustrate $\mathcal{A}_\delta(\eps)$ (compared to the linear case \cite{BCY}, the region of the resonant steps $\mathcal{A}_\delta^c(\eps) := \mathbb{R}^+\backslash \mathcal{A}_\delta(\eps)$  for fixed $\eps$ are doubled due to the cubic nonlinearity), we show in Figure \ref{fig:axis} for $\eps = 1$ and $\eps = 0.5$ with fixed $\delta = 0.15$.
\begin{figure}
	\centering
	\includegraphics[width=0.9\textwidth]{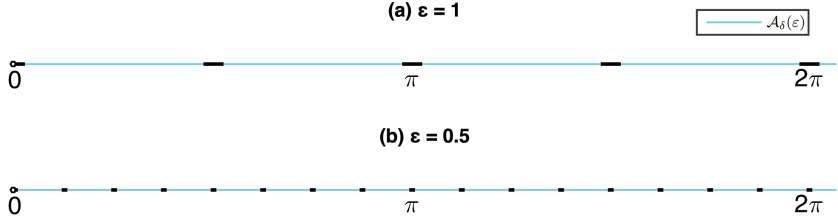}
	\caption{Illustration of the non-resonant time step $\mathcal{A}_\delta(\eps)$ with $\delta = 0.15$ for
		(a) $\eps = 1$ and (b) $\eps = 0.5$.}
	\label{fig:axis}
\end{figure}

For $\tau\in \mathcal{A}_\delta(\eps)$, we can derive improved uniform error bounds for $S_1$ as follows.
\begin{theorem}\label{thm:lie2}
	Let $\Phi^n(x)$ be the
	numerical approximation obtained from $S_1$ \eqref{eq:S1}.
	If the time step size $\tau$ is non-resonant, i.e. there exists $0 < \delta \leq 1$, such that $\tau\in \mathcal{A}_\delta(\eps)$,  then under the assumptions $(A)$ and $(B)$ with $m =1$,	we have an improved uniform error bound for small enough $\tau>0$
	\be
	\|{\bf e}^n(x)\|_{H^1}\lesssim_\delta \tau, \quad 0\le n\le\frac{T}{\tau}.
	\ee
\end{theorem}
\begin{proof} First of all, the assumptions of Theorem \ref{thm:lie} are satisfied in Theorem \ref{thm:lie2}, so we can directly use the results of Theorem \ref{thm:lie}. In particular, the numerical
solution $\Phi^n$ are bounded in $H^1$ as $\|\Phi^n\|_{H^1}\leq M_1+1$ \eqref{eq:indu} and Lemma \ref{lemma:lie} for local truncation error holds.

We start from \eqref{eq:uniform:s1}.
	The improved estimates rely on the cancellation phenomenon for the $\eta_2^k$ term in \eqref{eq:uniform:s1}.   From Lemma \ref{lemma:lie}, \eqref{eq:f2nlie}, \eqref{eq:g1} and \eqref{eq:g2}, we can write $\eta_2^k(x)$ as	
	\begin{align}\label{eq:S1_eta2}
	\eta_2^k(x):= &p_1(\tau)\mathcal{R}_{4,-}(\Phi_+^\eps(t_k),\Phi_-^\eps(t_k))-\overline{p_1(\tau)}\mathcal{R}_{4,+}(\Phi_+^\eps(t_k),\Phi_-^\eps(t_k))\\
&+p_2(\tau)\mathcal{R}_{2,-}(\Phi_+^\eps(t_k),\Phi_-^\eps(t_k))-\overline{p_2(\tau)}\mathcal{R}_{2,+}(\Phi_+^\eps(t_k),\Phi_-^\eps(t_k)),
	&\nonumber
	\end{align}
	where  $\mathcal{R}_{j,\pm}(\Phi_+^\eps,\Phi_-^\eps)$ ($j=2,4, \Phi_+^\eps,\Phi_-^\eps: \mathbb{R}\to \mathbb{C}^2$) are as follows
	\be\label{eq:mar:1}
	\begin{split}
&\mathcal{R}_{4,-}(\Phi_+^\eps,\Phi_-^\eps)=\Pi_-^\eps\left(\mathbf{g}_1(\Phi_+^\eps,\Phi_-^\eps)\Phi_+^\eps\right),
\quad \mathcal{R}_{4,+}(\Phi_+^\eps,\Phi_-^\eps)=\Pi_+^\eps\left(\overline{\mathbf{g}_1(\Phi_+^\eps,\Phi_-^\eps)}\Phi_-^\eps\right), \\
	& \mathcal{R}_{2,-}(\Phi_+^\eps,\Phi_-^\eps)=\Pi_+^\eps\left(\mathbf{g}_1(\Phi_+^\eps,\Phi_-^\eps)\Phi_+^\eps\right) + \Pi_-^\eps\left(\mathbf{g}_2(\Phi_+^\eps,\Phi_-^\eps)\Phi_+^\eps+\mathbf{g}_1(\Phi_+^\eps,\Phi_-^\eps)\Phi_-^\eps \right),\\
& \mathcal{R}_{2,+}(\Phi_+^\eps,\Phi_-^\eps)=\Pi_-^\eps\left(\overline{\mathbf{g}_1(\Phi_+^\eps,\Phi_-^\eps)}\Phi_-^\eps\right) + \Pi_+^\eps\left(\mathbf{g}_2(\Phi_+^\eps,\Phi_-^\eps)\Phi_-^\eps+\overline{\mathbf{g}_1(\Phi_+^\eps,\Phi_-^\eps)}\Phi_+^\eps \right),
	\end{split}
	\ee
	with $\mathbf{g}_1,\mathbf{g}_2$  given in  \eqref{eq:g1}-\eqref{eq:g2}  (Lemma \ref{lemma:lie}), and
	\begin{align}\label{eq:p12}
	p_1(\tau)=-i\left(\int_0^\tau e^{-\frac{4si}{\eps^2}}\,ds-\tau\right),\quad
		p_2(\tau)=-i\left(\int_0^\tau e^{-\frac{2si}{\eps^2}}\,ds-\tau\right).
	\end{align}
	It is obvious that $|p_1(\tau)|,|p_2(\tau)|\leq 2\tau$ and \eqref{eq:uniform:s1} implies that
	\begin{align}\label{eq:ebd:1}
	\|{\bf e}^{n+1}(x)\|_{H^1} \lesssim \tau + \tau\sum_{k=0}^n\|{\bf e}^k(x)\|_{H^1} +\tau\sum_{\sigma=\pm,j=2,4} \left\|\sum_{k = 0}^n e^{-\frac{i(n-k+1)\tau}{\eps^2}Q{^\eps}}\mathcal{R}_{j,\sigma}(\Phi_+^\eps(t_k),\Phi_-^\eps(t_k))\right\|_{H^1}.
	\end{align}
	To proceed, we introduce $\widetilde{\Phi}_\pm^\eps(t)$ as
	\be \label{eq:tildephi}
	\widetilde{\Phi}_\pm^\eps(t):=\tilde{\Phi}_\pm^\eps(t,x)=e^{\pm\frac{it}{\eps^2}}\Phi_\pm^\eps(t,x), \quad 0\leq t\leq T.
	\ee
	Since $\Phi(t,x)$ solves the NLDE \eqref{eq:NLDirac1d2d} (or \eqref{eq:NLDirac_nonmag}), noticing the properties of $Q^\eps$ as in \eqref{eq:Teps} and \eqref{eq:Deps} and the $L^2$ orthogonal projections $\Pi_{\pm}^\eps$, it is straightforward to compute that
	\be
	i\partial_t{\widetilde{\Phi}_\pm^\eps(t)}=\mathcal{D}^\eps \widetilde{\Phi}_{\pm}^\eps(t)+\Pi_{\pm}^\eps \left(e^{\mp\frac{it}{\eps^2}}\mathbf{F}(\Phi(t))\Phi(t)\right),
	\ee
	and the assumptions of Theorem \ref{thm:lie} would yield
	\be\label{eq:tildephi:bd}
	\|\tilde{\Phi}_{\pm}^\eps(\cdot)\|_{L^\infty([0,T];(H^3)^2)}\lesssim 1,\quad
	\|\partial_t\tilde{\Phi}_{\pm}^\eps(\cdot)\|_{L^\infty([0,T];(H^1)^2)}\lesssim 1.
	\ee
	Now, we can deal with the terms involving $\mathcal{R}_{j,\pm}$ ($j=2,4$) in \eqref{eq:mar:1}.

{\it For $\mathcal{R}_{4,-}$:}  By direct computation, we  get $\mathcal{R}_{4,-}(\Phi_+^\eps(t_k),\Phi_-^\eps(t_k))=e^{-\frac{3it_k}{\eps^2}}\mathcal{R}_{4,-}(\widetilde{\Phi}_+^\eps(t_k),\widetilde{\Phi}_-^\eps(t_k))$.
In view of \eqref{eq:dec:ev} and \eqref{eq:mar:1}, we have for $0\leq k\leq n\leq\frac{T}{\tau}-1$,
\begin{align}\label{eq:r41}
e^{-\frac{i(n-k+1)\tau}{\eps^2}Q^\eps}\mathcal{R}_{4,-}(\Phi_+^\eps(t_k),\Phi_-^\eps(t_k))
=e^{\frac{i(n+1-4k)\tau}{\eps^2}} e^{i(t_{n+1}-t_k)\mathcal{D}^\eps}\mathcal{R}_{4,-}(\widetilde{\Phi}_+^\eps(t_k),\widetilde{\Phi}_-^\eps(t_k)).
\end{align}
Denoting
\be
A(t):=A(t,x)=e^{-it\mathcal{D}^\eps}\mathcal{R}_{4,-}(\widetilde{\Phi}_+^\eps(t),\widetilde{\Phi}_-^\eps(t)),\quad 0\leq t\leq T,
\ee
and noticing that $\partial_tA(t)=-ie^{-it\mathcal{D}^\eps}\mathcal{D}^\eps\mathcal{R}_{4,-}(\widetilde{\Phi}_+^\eps(t),\widetilde{\Phi}_-^\eps(t))+e^{-it\mathcal{D}^\eps}\partial_t\mathcal{R}_{4,-}(\widetilde{\Phi}_+^\eps(t),\widetilde{\Phi}_-^\eps(t))$,
 we can derive from \eqref{eq:tildephi:bd} and the fact that $\mathcal{D}^\eps:(H^m)^2\to (H^{m-2})^2$ is uniformly bounded w.r.t $\eps$,
\begin{align}
\|A(t_k)-A(t_{k-1})\|_{H^1}\lesssim& \tau\left[\|\mathcal{R}_{4,-}(\widetilde{\Phi}_+^\eps(t_k),\widetilde{\Phi}_-^\eps(t_k))\|_{H^3}+\|\partial_t\mathcal{R}_{4,-}(\widetilde{\Phi}_+^\eps(t),\widetilde{\Phi}_-^\eps(t))\|_{L^\infty([0,T];(H^1)^2)}\right]\nonumber\\
\lesssim&\tau, \label{eq:bda} \quad 1\leq k\leq\frac{T}{\tau}.
\end{align}
Using \eqref{eq:bda}, \eqref{eq:r41},  $\|A(t)\|_{L^\infty([0,T];(H^1)^2)}\lesssim1$, the property that  $e^{it\mathcal{D}^\eps}$ preserves $H^1$ norm, summation by parts formula and triangle inequality, we have
\begin{align}\label{eq:sm1}
&\left\|\sum_{k = 0}^n e^{-\frac{i(n-k+1)\tau}{\eps^2}Q^\eps}\mathcal{R}_{4,-}(\Phi_+^\eps(t_k),\Phi_-^\eps(t_k))\right\|_{H^1}=\left\|\sum_{k = 0}^n e^{-\frac{i4k\tau}{\eps^2}}A(t_k)\right\|_{H^1}
\\
&\leq \left\|\sum_{k = 0}^{n-1}\theta_{k}(A(t_k)-A(t_{k+1}))\right\|_{H^1}
+\|\theta_nA(t_n)\|_{H^1}\lesssim \tau \left|\sum_{k = 0}^{n-1}\theta_{k}\right|+1,\nonumber
\end{align}	
with
\be\label{eq:ser1}
\theta_k=\sum_{j=0}^ke^{-\frac{i4j\tau}{\eps^2}}=\frac{1-e^{-\frac{i4(k+1)\tau}{\eps^2}}}{1-e^{-\frac{i4\tau}{\eps^2}}},\quad k\ge0,\quad \theta_{-1}=0.
\ee	
For $\tau\in\mathcal{A}_\delta(\eps)$	 \eqref{tau_range}, we have $|1-e^{-\frac{i4\tau}{\eps^2}}|=|2\sin(2\tau/\eps^2)|\ge2\delta$ and $|\theta_k|\leq\frac{2}{2\delta}=1/\delta$, and \eqref{eq:sm1} leads to
\be\label{eq:sum1}
\left\|\sum_{k = 0}^n e^{-\frac{i(n-k+1)\tau}{\eps^2}Q^\eps}\mathcal{R}_{4,-}(\Phi_+^\eps(t_k),\Phi_-^\eps(t_k))\right\|_{H^1}\lesssim \frac{n\tau+1}{\delta}\lesssim\frac{1}{\delta}.
\ee

{\it For $\mathcal{R}_{2,-}$:} Similar to the case $\mathcal{R}_{4,-}$ (slightly different), it is straightforward to show that
\begin{align}
e^{-\frac{i(n-k+1)\tau}{\eps^2}Q^\eps}\mathcal{R}_{2,-}(\Phi_+^\eps(t_k),\Phi_-^\eps(t_k))
=e^{\frac{i(n+1-2k)\tau}{\eps^2}}\left[ e^{-it_{n+1}\mathcal{D}^\eps}B(t_k)+e^{it_{n+1}\mathcal{D}^\eps}C(t_k)\right],
\end{align}
where
\begin{align}
B(t)=&e^{it\mathcal{D}^\eps}\Pi_+^\eps\left(\mathbf{g}_1(\widetilde{\Phi}_+^\eps(t),\widetilde{\Phi}_-^\eps(t))\widetilde{\Phi}^\eps_+(t)\right) ,\\
C(t)=&e^{-it\mathcal{D}^\eps}\Pi_-^\eps\left(\mathbf{g}_2(\widetilde{\Phi}_+^\eps(t),\widetilde{\Phi}_-^\eps(t))\widetilde{\Phi}_+^\eps(t)+\mathbf{g}_1(\widetilde{\Phi}_+^\eps(t),\widetilde{\Phi}_-^\eps(t))\Phi_-^\eps(t) \right).
\end{align}
$B(t)$ and $C(t)$ satisfy the same estimates as $A(t)$ \eqref{eq:bda}. Therefore, similar procedure  will give
\begin{align}\label{eq:sm2}
&\left\|\sum_{k = 0}^n e^{-\frac{i(n-k+1)\tau}{\eps^2}Q^\eps}\mathcal{R}_{2,-}(\Phi_+^\eps(t_k),\Phi_-^\eps(t_k))\right\|_{H^1}\leq
\left\|\sum_{k = 0}^n e^{-\frac{i2k\tau}{\eps^2}}B(t_k)\right\|_{H^1}+\left\|\sum_{k = 0}^n e^{-\frac{i2k\tau}{\eps^2}}C(t_k)\right\|_{H^1}
\\
&\lesssim \tau \left|\sum_{k = 0}^{n-1}\widetilde{\theta}^{k}\right|+1,\nonumber
\end{align}
with $\widetilde{\theta}_k=\sum_{j=0}^ke^{-\frac{i2j\tau}{\eps^2}}=\frac{1-e^{-\frac{i2(k+1)\tau}{\eps^2}}}{1-e^{\frac{-i2\tau}{\eps^2}}},\quad k\ge0,\quad \widetilde{\theta}_{-1}=0$.
For $\tau\in\mathcal{A}_\delta(\eps)$	 \eqref{tau_range}, we know $|1-e^{\frac{-i2\tau}{\eps^2}}|=|2\sin(\tau/\eps^2)|\ge|4\sin(2\tau/\eps^2)|\ge4\delta$ and $|\widetilde{\theta}_k|\leq\frac{2}{4\delta}=2/\delta$, which shows	
\begin{align}\label{eq:sum2}
&\left\|\sum_{k = 0}^n e^{-\frac{i(n-k+1)\tau}{\eps^2}Q^\eps}\mathcal{R}_{2,-}(\Phi_+^\eps(t_k),\Phi_-^\eps(t_k))\right\|_{H^1}
\lesssim \tau \left|\sum_{k = 0}^{n-1}\widetilde{\theta}^{k}\right|+1\lesssim\frac{1}{\delta}.
\end{align}

{\it For $\mathcal{R}_{4,+}$ and $\mathcal{R}_{2,+}$:} It is easy to see that the $\mathcal{R}_{4,+}$ and $\mathcal{R}_{2,+}$ terms in \eqref{eq:ebd:1} can be bounded exactly the same as
the $\mathcal{R}_{4,-}$ and $\mathcal{R}_{2,-}$ terms, respectively.

Finally, combining \eqref{eq:ebd:1}, \eqref{eq:sum1}, \eqref{eq:sum2} and above observations, we have for $\tau\in\mathcal{A}_\delta(\eps)$,
\be
\|{\bf e}^{n+1}(x)\|_{H^1} \lesssim \frac{\tau}{\delta} + \tau\sum_{k=0}^n\|{\bf e}^k(x)\|_{H^1} ,\quad 0\leq n\leq \frac{T}{\tau}-1,
\ee
and discrete Gronwall inequality yields $\|{\bf e}^{n+1}(x)\|_{H^1} \lesssim \frac{\tau}{\delta}$ ($0\leq n\leq \frac{T}{\tau}-1$) for small enough $\tau\in\mathcal{A}_\delta(\eps)$. The proof is completed.
\end{proof}


\subsection{Numerical results}
To verify our error bounds in Theorems \ref{thm:lie} and \ref{thm:lie2}, we show a numerical example here. In this example and all the numerical examples later, we always use Fourier pseudospectral method for spatial discretization.

As a common practice when applying the Fourier pseudospectral method, in our numerical simulations, we truncate the whole space onto a sufficiently large bounded domain $\Omega = (a, b)$, and assume periodic boundary conditions. The mesh size is chosen as $h := \triangle x = \frac{b - a}{M}$ with $M$ being an even positive integer. Then the grid points can be denoted as $x_j := a + jh$, for $j = 0, 1, ..., M$.

In this example, we choose the electric potential $V(x)\equiv 0$. For the nonlinearity \eqref{eq:F}, we take $\lambda_1 = 1$, $\lambda_2 = 0$, i.e.
\be\label{eq:numer_F}
\mathbf{F}(\Phi) = (\Phi^*\sigma_3\Phi)\sigma_3,
\ee
and the initial data  $\Phi_0=(\phi_1,\phi_2)$  in \eqref{eq:initial} is given as
\be\label{eq:numer_ini}
\phi_1(0, x) =  e^{-\frac{x^2}{2}}, \quad \phi_2(0, x) = e^{-\frac{(x - 1)^2}{2}}, \quad x\in\mathbb{R}.
\ee

As only the temporal errors are concerned in this paper, during the computation, the spatial mesh size is always set to be $h = \frac{1}{16}$ so that the spatial errors are negligible.

We first take resonant time steps, that is, for small enough chosen $\eps$, there is a positive $k_0$, such that $\tau = \frac{1}{2}k_0\eps^2\pi$, to check the error bounds in Theorem \ref{thm:lie}. The bounded computational domain is taken as $\Omega = (-32, 32)$, i.e., $a = -32$ and $b = 32$.
Because  the  exact solution is unknown, for comparison, we use a numerical `exact' solution generated by the second-order time-splitting method ($S_2$), which will be introduced later, with a very fine time step size $\tau_e = 2\pi\times10^{-6}$.

\begin{table}[htp]
	\def\temptablewidth{1\textwidth}
	\caption{Discrete $H^1$ temporal errors $e^{\eps, \tau}(t = 2\pi)$ for the wave function with resonant time step size, $S_1$ method. }
	{\rule{\temptablewidth}{1pt}}
	\begin{tabular*}{\temptablewidth}{@{\extracolsep{\fill}}ccccccc}
		$e^{\eps, \tau}(t = 2\pi)$  & $\tau_0 = \pi/4$ & $\tau_0 / 4$ & $\tau_0 / 4^2$ & $\tau_0 / 4^3$
		& $\tau_0 / 4^4$ & $\tau_0 / 4^5$ \\ \hline
		$\eps_0=1$ & 4.18 &	\textbf{7.09E-1} &	1.69E-1 &	4.17E-2 &	1.04E-2 &	2.59E-3\\
		order & -- & \textbf{1.28} &	1.04 &	1.01 &	1.00 &	1.00\\\hline
		$\eps_0/2$ & 2.54 &	6.37E-1 &	\textbf{1.44E-1} &	3.55E-2 &	8.84E-3 &	2.21E-3\\
		order & -- & 1.00 &	\textbf{1.07} &	1.01 &	1.00 &	1.00\\\hline
		$\eps_0/2^2$ & 2.25 & 1.15 &	1.47E-1 &	\textbf{3.53E-2} &	8.73E-3 &	2.18E-3\\
		order & -- & 0.49 &	1.48 &	\textbf{1.03} &	1.01 &	1.00\\\hline
		$\eps_0/2^3$ & 2.29 &	6.69E-1 &	6.56E-1 &	3.62E-2 &	\textbf{8.84E-3} &	2.20E-3\\
		order & -- & 0.89 &	0.01 &	2.09 &	\textbf{1.02} &	1.00\\\hline
		$\eps_0/2^4$ & 2.32 &	5.33E-1 &	3.24E-1	& 3.49E-1 &	8.98E-3	 & \textbf{2.22E-3}\\
		order & -- & 1.06 &	0.36 &	-0.05 &	2.64 &	\textbf{1.01}\\\hline
		$\eps_0/2^5$ & 2.34 &	\textbf{5.29E-1} &	1.76E-1 &	1.70E-1 &	1.79E-1 &	2.24E-3\\
		order & -- & \textbf{1.07} &	0.79 &	0.03 &	-0.04 &	3.16 \\\hline
		$\eps_0/2^7$ & 2.35 &	5.57E-1 &	\textbf{1.30E-1} &	4.46E-2 &	4.28E-2 &	4.49E-2\\
		order & -- & 1.04 &	\textbf{1.05} &	0.77 &	0.03 &	-0.03\\\hline
		$\eps_0/2^9$ & 2.35 &	5.68E-1 &	1.38E-1 &	\textbf{3.26E-2} &	1.12E-2 &	1.07E-2\\
		order & -- & 1.02 &	1.02 &	\textbf{1.04} &	0.77 &	0.03\\\hline
		$\eps_0/2^{11}$ & 2.35 &	5.71E-1 &	1.41E-1 &	3.45E-2 &	\textbf{8.14E-3} &	2.80E-3\\
		order & -- & 1.02 &	1.01 &	1.02 & 	\textbf{1.04} &	0.77 \\\hline
		$\eps_0/2^{13}$ & 2.35 &	5.72E-1 &	1.42E-1 &	3.53E-2 &	8.64E-3 &	\textbf{2.04E-3}\\
		order & -- & 1.02 &	1.00 &	1.00 &	1.02 &	\textbf{1.04} \\\hline\hline
		$\max\limits_{0<\eps\leq 1}e^{\eps, \tau}(t=2\pi)$ & 4.18 &	1.15 &	6.56E-1 &	3.49E-1 &	1.79E-1 &	9.07E-2\\
		order & -- & 0.93 &	0.40 &	0.45 &	0.48 &	0.49
	\end{tabular*}
	{\rule{\temptablewidth}{1pt}}
	\label{table:NLDirac_S1}
\end{table}

To display the numerical results, we introduce the discrete $H^1$ errors of the numerical solution. Let $\Phi^n = (\Phi_0^n, \Phi_1^n, ..., \Phi_{M-1}^n,\Phi_M^n)^T$ be the numerical solution obtained by a numerical method with given $\eps$, time step size $\tau$ as well as the fine mesh size
$h$ at time $t = t_n$, and $\Phi(t, x)$ be the exact solution, then the discrete $H^1$ error is defined as
\be\label{eq:discrete_H1}
e^{\eps, \tau}(t_n) = \|\Phi^n - \Phi(t_n, \cdot)\|_{H^1} = \sqrt{h\sum_{j = 0}^{M - 1}|\Phi(t_n, x_j) - \Phi_j^n|^2 + h\sum_{j=0}^{M-1}|\Phi'(t_n, x_j) - (\Phi')_j^n|^2},
\ee
where
\be
(\Phi')_j^n = i\sum_{l = -M/2}^{M/2-1}\mu_l\widehat{\Phi}_l^ne^{i\mu_l(x_j-a)}, \quad j = 0, 1, ..., M-1,
\ee
with $\mu_l$, $\widehat{\Phi}_l^n\in\mathbb{C}^2$ defined as
\be
\mu_l = \frac{2l\pi}{b-a}, \quad \widehat{\Phi}_l^n = \frac{1}{M}\sum_{j=0}^{M-1}\Phi_j^ne^{-i\mu_l(x_j-a)}, \quad l = -\frac{M}{2}, ..., \frac{M}{2} - 1,
\ee
and $\Phi'(t_n, x_j)$ is defined similarly. Then $e^{\eps, \tau}(t_n)$ should be close to the $H^1$ errors in Theorem \ref{thm:lie} for fine spatial mesh sizes $h$.

Table \ref{table:NLDirac_S1} shows the  temporal errors $e^{\eps, \tau}(t = 2\pi)$ with different $\varepsilon$ and time step size $\tau$ for $S_1$.

The last two rows of Table \ref{table:NLDirac_S1} show the largest error of each column for fixed $\tau$. The errors exhibit $1/2$ order convergence, which coincides well with Theorems \ref{thm:lie}.
More specifically, we can observe when $\tau\gtrsim \eps$ (below the lower bolded diagonal line), there is first order convergence, which agrees with the error bound $\|\Phi(t_n, x) - \Phi^n(x)\|_{H^1}\lesssim \tau +\eps$. When $\tau\lesssim\eps^2$ (above the upper bolded diagonal line), there is also first order convergence, which matches the other error bound $\|\Phi(t_n, x) - \Phi^n(x)\|_{H^1}\lesssim \tau +\tau/ \eps$.

To support the improved uniform error bound in Theorem \ref{thm:lie2}, we further test the discrete errors using non-resonant time steps, i.e., we choose $\tau\in\mathcal{A}_\delta(\eps)$ for some given $\eps$ and fixed $0<\delta \leq1$. In this case, the bounded computational domain is set as $\Omega = (-16, 16)$.

For comparison, the numerical `exact' solution is computed by the second-order time-splitting method ($S_2$) with a very small time step size $\tau_e = 8\times10^{-6}$. 

Figure \ref{fig:NLDirac_S1_nrs} shows the errors $e^{\eps, \tau}(t = 4)$ with different $\varepsilon$ and time step size $\tau$ for $S_1$.

\begin{figure}
	\centering
	\includegraphics[width=0.45\textwidth]{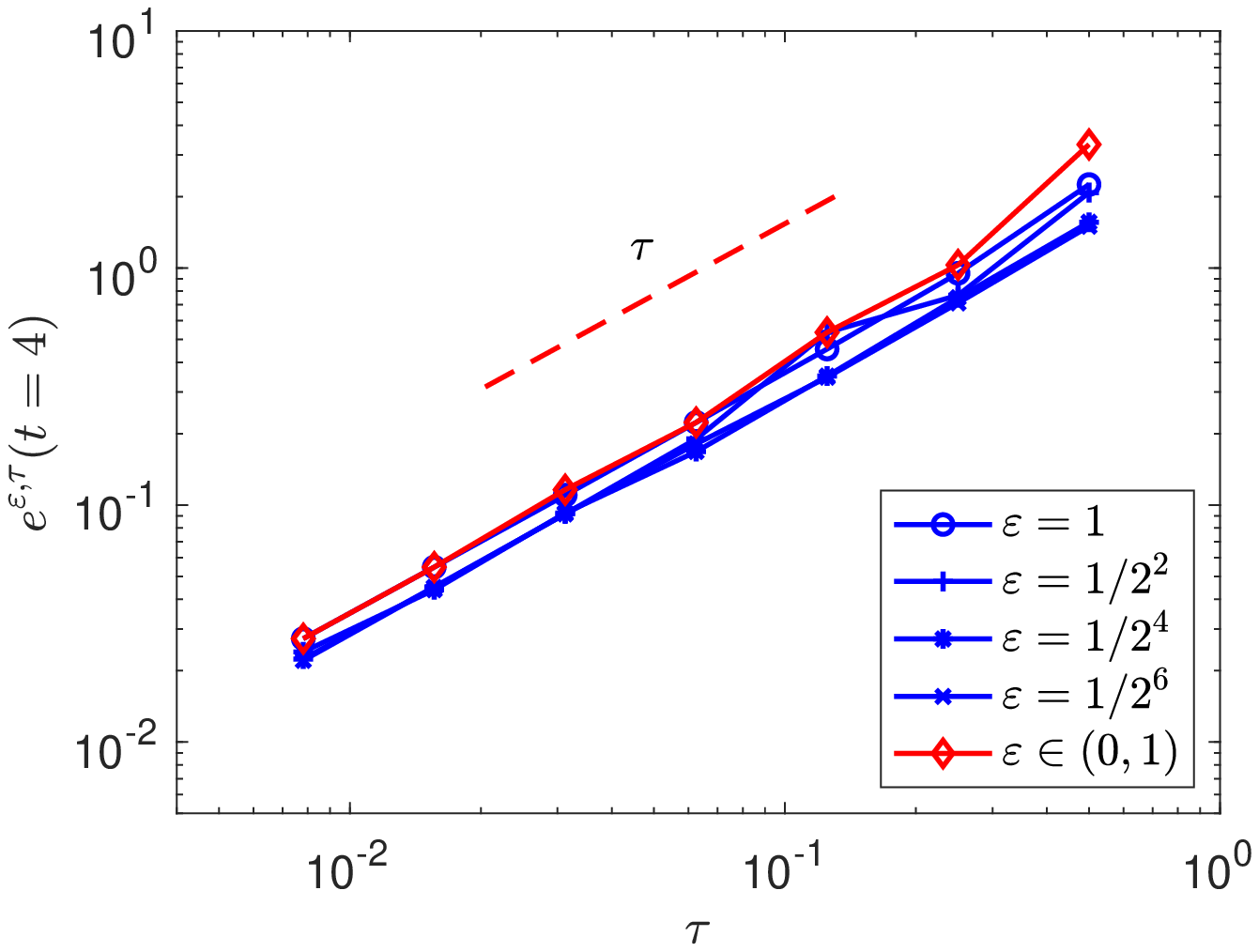}
	\includegraphics[width=0.45\textwidth]{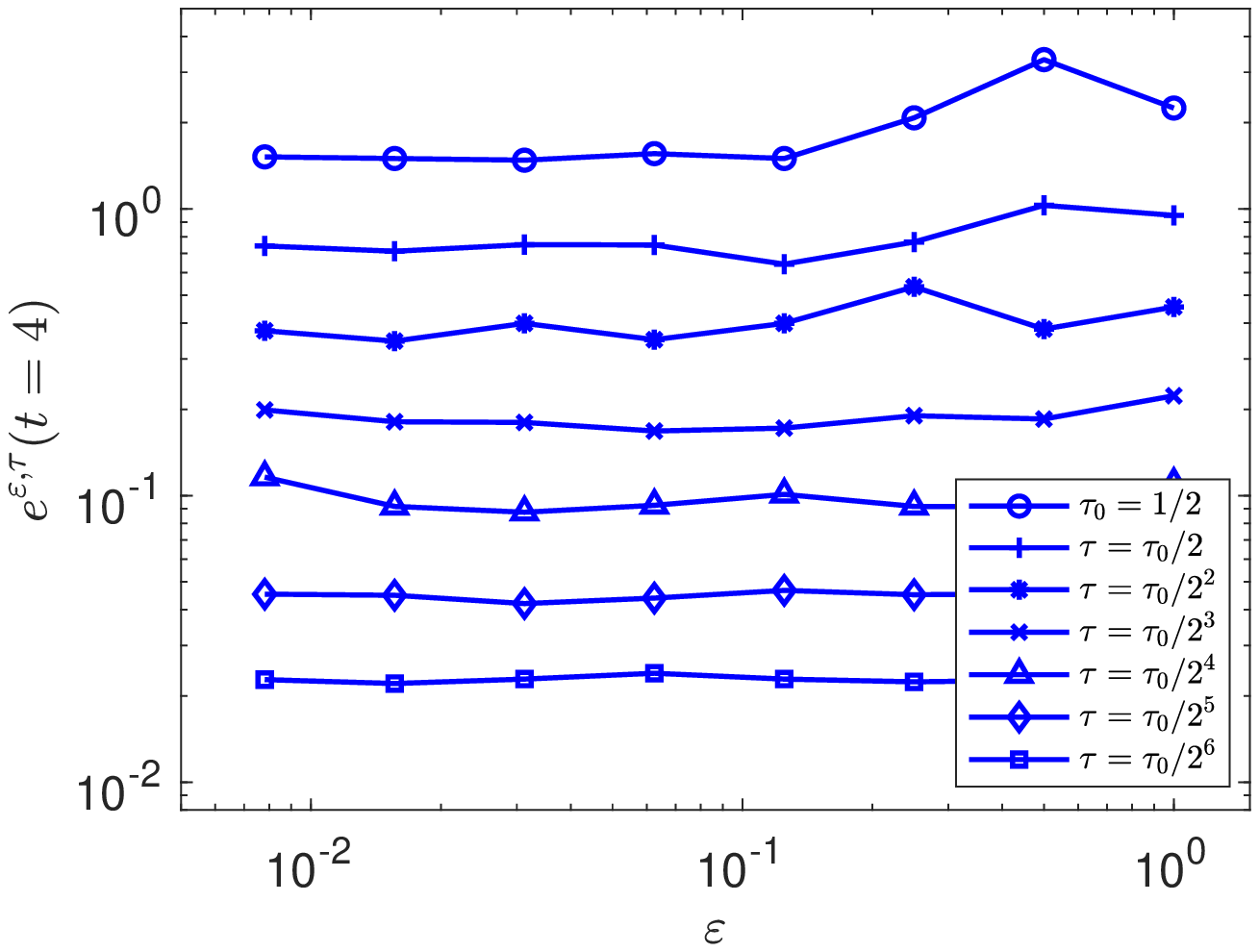}
	\caption{The discrete $H^1$ error $e^{\eps, \tau}(t = 4)$ with respect to $\tau$ and $\eps$ with non-resonant time step sizes, $S_1$ method.}
	\label{fig:NLDirac_S1_nrs}
\end{figure}

From  the left part of Fig. \ref{fig:NLDirac_S1_nrs}, we could see that for each $\eps\in (0, 1]$, there is always first order convergence in $\tau$ for non-resonant time steps. From the right part, we find  that for fixed time step size $\tau$, i.e., for each line in the figure, the error $e^{\eps,\tau}(t = 4)$ does not change much with different $\eps$. This verifies the temporal uniform first order convergence for $S_1$ with non-resonant time step size, as stated in Theorem \ref{thm:lie2}.

Through the results of this example, we successfully validate the uniform error bounds for $S_1$ in Theorems \ref{thm:lie} \& \ref{thm:lie2}.

\section{Extension to the second-order splitting method}\setcounter{equation}{0}
In this section, we extend the results in the previous section
to the second-order Strang splitting method.

Applying the discrete-in-time second-order splitting (Strang splitting, $S_2$) to \eqref{eq:NLDirac_nonmag}, we have the numerical method as \cite{BCJY, Strang}
\be\label{eq:S2}
\Phi^{n + 1}(x) = e^{-\frac{i\tau}{2\eps^2}Q^\eps}e^{-i\tau \left[V(x) + \mathbf{F}\left(e^{-\frac{i\tau}{2\eps^2}
Q^\eps}\Phi^n(x)\right)\right]}e^{-\frac{i\tau}{2\eps^2}
Q^\eps}\Phi^n(x),
\ee
with  $\Phi^0(x) = \Phi_0(x)$. We write the numerical propagator for $S_2$
as $\Phi^{n+1}(x):=S_{n,\tau}^{\text{Str}}(\Phi^n)$.

\subsection{Uniform error bounds}

For the numerical approximation $\Phi^n(x)$ obtained from $S_2$ \eqref{eq:S2}, we introduce the error function as in $S_1$
\be\label{eq:error2}
{\bf e}^n(x) = \Phi(t_n, x) - \Phi^n(x), \quad 0\leq n\leq \frac{T}{\tau},
\ee
and  the following uniform error bounds hold.

\begin{theorem}\label{thm:strang}
	Let $\Phi^n(x)$ be the
	numerical approximation obtained from $S_2$ \eqref{eq:S2}, then
	under the assumptions $(A)$ and $(B)$ with $m = 2$, there exists $0<\tau_0\leq1$ independent of $\eps$ such that
 the following error estimates hold for $0<\tau<\tau_0$,
	\be
	\|{\bf e}^n(x)\|_{H^1}\lesssim \tau^2+\eps,\quad \|{\bf e}^n(x)\|_{H^1}\lesssim \tau^2+\tau^2/\eps^3, \quad 0\le n\le\frac{T}{\tau}.
	\ee
	As a result, there is a uniform error bound for $S_2$ for $\tau>0$ small enough
	\be
	\|{\bf e}^n(x)\|_{H^1}\lesssim \tau^2 + \max_{0<\eps\leq 1}\min\{\eps, \tau^2/\eps^3\}\lesssim \sqrt{\tau},  \quad 0\leq n\leq\frac{T}{\tau}.
	\ee
\end{theorem}
\begin{proof}As the proof of the theorem is not difficult to establish by combining the techniques used in proving Theorem \ref{thm:lie} and the ideas in the proof of the uniform error bounds for $S_2$ in the linear case \cite{BCY}, we only give the outline of the proof here. For simplicity, we
assume $V(x)\equiv 0$ and denote $\Phi(t) := \Phi(t, x)$, $\Phi^n := \Phi^n(x)$ in short. Similar to the $S_1$ case, the $H^1$ bound of the numerical solution $\Phi^n$ is needed and can be done by using mathematical induction. For simplicity, we will assume the $H^1$ bound of $\Phi^n$   as in \eqref{eq:indu}.
	
\textbf{Step 1.} Use Taylor expansion and Duhamel's principle repeatedly to represent the `local truncation error' $\eta^n(x)=\Phi(t_{n+1})-e^{-\frac{i\tau}{2\eps^2}Q^\eps}e^{-i\tau\mathbf{F}(e^{-\frac{i\tau}{2\eps^2}Q^\eps}\Phi^n)}e^{-\frac{i\tau}{2\eps^2}Q^\eps}\Phi(t_n)$ \cite{BCJY,Lubich} as
\begin{equation*}
\eta^n(x)=	
e^{-\frac{i\tau}{\eps^2}Q^\eps}\left[\int_0^\tau (f^n(s)+h^n(s))\,ds-\tau f^n\left(\frac{\tau}{2}\right)
-\int_0^\tau\int_0^sg^n(s,w)\,dwds+\frac{\tau^2}{2}g^n\left(\frac{\tau}{2},\frac{\tau}{2}\right)\right]+R^n(x),\end{equation*}
where $\|R^n(x)\|_{H^1}\lesssim\tau^3+\tau\|\mathbf{e}^n(x)\|_{H^1}$, $f^n(s)$ is the same as that in Lie splitting $S_1$ case \eqref{eq:fns} and
\begin{align}
\label{eq:hn:S2}
&h^n(s)=-ie^{\frac{is}{\eps^2}Q^\eps}\left[\left(
\mathbf{F}\left(\Phi(t_{n}+s)\right)-\mathbf{F}\left(e^{-\frac{is}{\eps^2}Q^\eps}\Phi(t_n)\right)\right)e^{-\frac{is}{\eps^2}Q^\eps}\Phi(t_n)\right],\quad 0\leq s\leq\tau,\\
&g^n(s,w)=e^{\frac{is}{\eps^2}Q^\eps}\left(\mathbf{F}(e^{-\frac{is}{\eps^2}Q^\eps}\Phi(t_n))e^{-\frac{i(s-w)}{\eps^2}Q^\eps}\left(\mathbf{F}(e^{-\frac{is}{\eps^2}Q^\eps}\Phi(t_n))e^{-\frac{iw}{\eps^2}Q^\eps}\Phi(t_n)\right)\right),\;0\leq s,w\leq \tau.\label{eq:gn:S2}
\end{align}

\textbf{Step 2.}  For $h^n(s)$, using Duhamel's principle to get
\begin{align}
\Phi(t_n+s)=&e^{-\frac{is}{\eps^2}Q^\eps}\Phi(t_n)-ie^{-\frac{is}{\eps^2}Q^\eps}\int_0^sf^n(w)\,dw+O(s^2)\\
=&\phi^n(s)-is\mathbf{F}(\phi^n(s))\phi^n(s)-\hat{f}^n(s)+O(s^2),\nonumber
\end{align}
where $\phi^n(s)=e^{-\frac{is}{\eps^2}Q^\eps}\Phi(t_n)$, $\hat{f}^n(s)=ie^{-\frac{is}{\eps^2}Q^\eps}\int_0^s(f^n(w)-f^n(s))\,dw$, and we could find
\begin{align*}
\mathbf{F}\left(\Phi(t_{n}+s)\right)-\mathbf{F}\left(e^{-\frac{is}{\eps^2}Q^\eps}\Phi(t_n)\right)
=-2\lambda_1\text{Re}\left((\phi^n(s))^*\sigma_3\hat{f}^n(s)\right)\sigma_3-2\lambda_2\text{Re}\left((\phi^n(s))^*\hat{f}^n(s)\right)I_2+O(s^2).
\end{align*}
Recalling $\hat{f}^n(s)=O(s)$ and  \eqref{eq:fn11}, we get $f^n(s)-f^n(w)=f^{n}_2(s)-f_2^n(w)+O(s)$ with $f_2^n(s)$ given in \eqref{eq:f2nlie}. Finally, under the assumption of Theorem \ref{thm:strang}, expanding $e^{-\frac{isQ^\eps}{\eps^2}}\Phi(t_n)=e^{-\frac{is}{\eps^2}}\Phi_+^\eps(t_n)+e^{\frac{is}{\eps^2}}\Phi_-^\eps(t_n)+O(s)$,  we can write the $h^n(s)$ term as
\begin{align}
\int_0^\tau h^n(s)\,ds=\zeta_1^n(x)+\kappa_1^n(x), \quad \|\kappa_1^n(x)\|_{H^1}\lesssim \tau^3,
\end{align}
with $\zeta_1^n(x)$ given as
\[\zeta_1^n(x):= 2i\int_0^\tau  e^{\frac{is}{\eps^2}Q^\eps}\left[\left(
\lambda_1\text{Re}\left((\phi^n(s))^*\sigma_3\hat{f}^n(s)\right)\sigma_3
+\lambda_2\text{Re}\left((\phi^n(s))^*\hat{f}^n(s)\right)I_2\right)
\phi^n(s)\right]\,ds.\]
By taking $e^{-\frac{isQ^\eps}{\eps^2}}\approx e^{-\frac{is}{\eps^2}}\Pi_++e^{\frac{is}{\eps^2}}\Pi_-$,  it can be proved that $\|\zeta_1^n(x)\|_{H^1}\lesssim \min\{\tau^2\eps,\frac{\tau^3}{\eps}\}$.

Similarly,  $g^n(s, w)$ can be written as \be
 g^n(s, w) = \mathcal{G}_1^n(s, w) + \mathcal{G}_2^n(s,w)+\mathcal{G}_3^n(s,w) ,
\ee
where $\|\mathcal{G}_3^n(s,w)\|_{H^1}\lesssim\tau$,  the oscillatory term (in time) $\mathcal{G}_1^n(s, w) $ simplifies $g^n(s,w)$ by using $e^{-\frac{isQ^\eps}{\eps^2}}\approx e^{-\frac{is}{\eps^2}}\Pi_++e^{\frac{is}{\eps^2}}\Pi_-$ and removing the non-oscillatory terms as in \eqref{eq:tfn},  $\mathcal{G}_2^n(s,w)=\mathcal{G}_2^n(0,0)$ is the non-oscillatory term  ($s,w$ independent) similar to \eqref{eq:tfn}, $\|\mathcal{G}_1^n(s,w)\|_{H^1}\lesssim\eps$. We can prove $\|\partial_s\mathcal{G}_1^n(s,w)\|_{H^1}\lesssim1/\eps$, $\|\partial_w\mathcal{G}_1^n(s,w)\|_{H^1}\lesssim1/\eps$.

Lastly,  $f^n(s)$ can be decomposed as
 \be
 f^n(s) = \mathcal{F}_1^n(s) + \mathcal{F}_2^n(s)+ \mathcal{F}_3^n(s),
\ee
where $\|\mathcal{F}_3^n(s)\|_{H^1}\lesssim\tau^2$, the oscillatory term (in time) $\mathcal{F}_1^n(s)$ simplifies $f^n(s)$ by using $e^{-\frac{isQ^\eps}{\eps^2}}=e^{-\frac{is}{\eps^2}}(I_2-is\mathcal{D}^\eps)\Pi_++e^{\frac{is}{\eps^2}}(I_2+is\mathcal{D}^\eps)\Pi_-+O(s^2)$ and removing the non-oscillatory terms as in \eqref{eq:tfn},  $\mathcal{F}_2^n(s)=\mathcal{F}_2^n(0)$ is the non-oscillatory term ($s$ independent) similar to \eqref{eq:tfn}. We can prove  $\|\mathcal{F}_1^n(s)\|_{H^1}\lesssim\eps$, $\|\partial_s\mathcal{F}_1^n(s)\|_{H^1}\lesssim1/\eps$, $\|\partial_{ss}\mathcal{F}_1^n(s)\|_{H^1}\lesssim1/\eps^3$.

Denote
\begin{align}\label{eq:zeta}
&\zeta_{2}^n(x) = \left(\int_0^\tau \mathcal{F}_1^n(s)\,ds-\tau \mathcal{F}_1^n(\tau/2)\right),\quad \zeta_{3}^n(x) = \left(\int_0^\tau\int_0^s\mathcal{G}_1^n(s,w)\,dwds-\frac{\tau^2}{2}\mathcal{G}_1^n(\tau/2,\tau/2)\right),
\end{align}
and we have
\be
\eta^n(x)=e^{-\frac{i\tau}{\eps^2}Q^\eps}\left[\zeta_1^n(x)+\zeta_{2}^n(x)-\zeta_3^n(x)\right]
+\kappa^n(x),
\ee
where $\kappa^n(x)=R^n(x)+e^{-\frac{i\tau}{\eps^2}Q^\eps}\left(\kappa_1^n(x)
+\int_0^\tau \mathcal{F}_3^n(s)\,ds-\tau \mathcal{F}_3^n\left(\frac{\tau}{2}\right)
-\int_0^\tau\int_0^s\mathcal{G}_3^n(s,w)\,dwds+\frac{\tau^2}{2}\mathcal{G}_3^n\left(\frac{\tau}{2},\frac{\tau}{2}\right)\right)$ and $\|\kappa^n(x)\|_{H^1}\lesssim \tau^3+\tau\|\mathbf{e}^n(x)\|_{H^1}$.

Following the idea in $S_1$ case \eqref{eq:eta:lie}, we have the error equation for $S_2$
\be\label{eq:rec:2}
{\bf e}^{n+1}(x) = e^{-\frac{i\tau}{\eps^2}Q^\eps}{\bf e}^n(x) + \zeta_1^n(x) + \zeta_2^n(x) - \zeta_3^n(x)+\kappa^n(x)+\widetilde{L}_n(\mathbf{e}^n(x)), \quad 0\leq n\leq \frac{T}{\tau} - 1,
\ee
where $\widetilde{L}_n{\mathbf{e}^n}(x)=e^{-\frac{i\tau}{2\eps^2}Q^\eps}\left(e^{-i\tau\mathbf{F}\left(e^{-\frac{i\tau}{2\eps^2}Q^\eps}\Phi^n\right)}-I_2\right)e^{-\frac{i\tau}{2\eps^2}Q^\eps}$,  and $\|\widetilde{L}_n{\mathbf{e}^n(x)}\|_{H^1}\leq e^{c_{M_1}\tau}\|\mathbf{e}^n(x)\|_{H^1}$ ($c_{M_1}$ depends on $M_1$). For $0\leq n\leq\frac{T}{\tau}-1$, we would  have (following \eqref{eq:ebd:1}),
\be\label{eq:eta:str}
\|{\bf e}^{n+1}(x)\|_{H^1} \lesssim \tau^2+ \tau\sum_{k=0}^n\|{\bf e}^k(x)\|_{H^1} +\sum_{j=1,2,3} \left\|\sum_{k = 0}^n e^{-\frac{i(n-k+1)\tau}{\eps^2}Q^\eps}\zeta_j^k(x)\right\|_{H^1}.
\ee

Under the hypothesis of Theorem \ref{thm:strang}, we have
\begin{align*}
&\|\mathcal{F}_1^n(s)\|_{H^1}\lesssim \eps, \quad \|\partial_s\mathcal{F}_1^n(s)\|_{H^1} \lesssim \eps/\eps^2 = 1/\eps, \quad \|\partial_{ss}\mathcal{F}_1^n(s)\|_{H^1}\lesssim 1 / \eps^3, \quad 0\leq s\leq\tau;\\
&\|\mathcal{G}_1^n(s, w)\|_{H^1}\lesssim\eps, \quad \|\partial_s\mathcal{G}_1^n(s, w)\|_{H^1}\lesssim 1 / \eps, \quad \|\partial_w\mathcal{G}_1^n(s, w)\|_{H^1}\lesssim 1 / \eps, \quad 0\leq s, w\leq\tau,
\end{align*}
which together with \eqref{eq:zeta} gives $\|\zeta_2^n(x)\|_{H^1}\lesssim\min\{\eps\tau,\tau^3/\eps^3\}$ and $\|\zeta_3^n(x)\|_{H^1}\lesssim\min\{\eps\tau^2,\tau^3/\eps\}$.
Since $ \|\zeta_1^n(x)\|_{H^1}\lesssim \min\{\tau^2\eps,\frac{\tau^3}{\eps}\}$, we derive from \eqref{eq:eta:str} that
\begin{align}
\|{\bf e}^{n+1}(x)\|_{H^1} \lesssim &\tau^2+ \tau\sum_{k=0}^n\|{\bf e}^k(x)\|_{H^1} +n\min\{\eps\tau^2, \tau^3/\eps\}+ \left\|\sum_{k = 0}^n e^{-\frac{i(n-k+1)\tau}{\eps^2}Q^\eps}\zeta_2^k(x)\right\|_{H^1}\label{eq:etag}\\
\lesssim &\tau^2 +n\min\{\eps\tau, \tau^3/\eps^3\} + \tau\sum_{k=0}^n\|{\bf e}^k(x)\|_{H^1}, \quad 0\leq n\leq\frac{T}{\tau} - 1.\nonumber
\end{align}
The discrete Gronwall's inequality gives the desired results in Theorem \ref{thm:strang} with the help of mathematical induction.
\end{proof}

\vspace{10pt}
For non-resonant time steps, i.e., for $\tau\in\mathcal{A}_\delta(\eps)$, similar to $S_1$, we can derive improved uniform error bounds for $S_2$ as shown in the following theorem.
\begin{theorem}\label{thm:strang2}
	Let $\Phi^n(x)$ be the
	numerical approximation obtained from $S_2$ \eqref{eq:S2}.
	If the time step size $\tau$ is non-resonant, i.e. there exists $0 < \delta \leq 1$, such that $\tau\in \mathcal{A}_\delta(\eps)$, then under the assumptions $(A)$ and $(B)$ with $m = 2$,
	the following two error estimates hold for small enough $\tau>0$
	\be
	\|{\bf e}^n(x)\|_{H^1}\lesssim_\delta \tau^2+\tau\eps,\quad \|{\bf e}^n(x)\|_{H^1}\lesssim_\delta \tau^2+\tau^2/\eps, \quad 0\le n\le\frac{T}{\tau}.
	\ee
	As a result, there is an improved uniform error bound for $S_2$ when $\tau>0$ is small enough
	\be
	\|{\bf e}^n(x)\|_{H^1}\lesssim_\delta \tau^2 + \max_{0<\eps\leq 1}\min\{\tau\eps, \tau^2/\eps\}\lesssim_\delta \tau^{3/2}, \quad 0\leq n\leq\frac{T}{\tau}.
	\ee
\end{theorem}
\begin{proof}
	As the proof is extended from the techniques used for $S_1$ and the proof for improved uniform error bounds for $S_2$ in the linear case \cite{BCY}, here we just show the outline of the proof for brevity.
	
	We start from \eqref{eq:etag}. Following the strategy in the $S_1$ case, the key idea is to extract the leading terms from $\Phi(t,x)$ as \eqref{eq:tildephi} for estimating $\zeta_2^n(x)$, and the computations are more or less the same.	Recalling \eqref{eq:zeta} , noticing $\mathcal{F}_1^n(s)$ is similar to $f_2^n(s)$ \eqref{eq:f2nlie} and $\|\zeta_2^n(x)\|_{H^1}\lesssim\min\{\eps\tau,\tau^2/\eps\}$,   following the computations in the proof of Theorem \ref{thm:lie2}, we would get for $0\leq n\leq\frac{T}{\tau}-1$ and $\tau\in A_\delta(\eps)$,
	\be
	\left\|\sum_{k = 0}^n e^{-\frac{i(n-k+1)\tau}{\eps^2}Q^\eps}\zeta_2^k(x)\right\|_{H^1}
	\lesssim \sum_{k=0}^n\frac{1}{\delta}\tau \min\{\eps\tau,\tau^2/\eps\}
	\lesssim\frac{1}{\delta}\min\{\eps\tau,\tau^2/\eps\},
	\ee
	and the conclusions of Theorem \ref{thm:strang2} hold by applying the discrete Gronwall inequality
to \eqref{eq:etag}.
\end{proof}
\\


\subsection{Numerical results}
In this subsection, we use a numerical example to validate our uniform error bounds in Theorems \ref{thm:strang} and \ref{thm:strang2}.

In the example, we  choose the nonlinearity and the initial values as \eqref{eq:numer_F} and \eqref{eq:numer_ini}. In order to show that the error estimates still hold for $V\neq 0$,  here we take the electric potential
\begin{equation}
    V(x) = \frac{x-1}{x^2+1}.
\end{equation}

We first test the errors for resonant time steps, that is, for small enough chosen $\eps$, there is a positive $k_0$, such that $\tau = \frac{1}{2}k_0\eps^2\pi$, to check the error bounds in Theorem \ref{thm:strang}.
In this case, the bounded computational domain is taken as $\Omega = (-32, 32)$. 
The numerical `exact' solution is generated by $S_2$ with a very fine time step size $\tau_e = 2\pi\times10^{-6}$.

The discrete $H^1$ error $e^{\eps,\tau}(t_n)$ used to show the results is defined in \eqref{eq:discrete_H1}. It should be close to the $H^1$ errors in Theorems \ref{thm:strang} here. In addition, we test the performance of $S_2$ in approximating the physical observables including probability density, current density, and energy. The discrete $l^1$ error for probability density is defined as
\be\label{eq:err_density}
e_\rho^{\eps, \tau}(t_n) = \|\rho^n - \rho(t_n, \cdot)\|_{l^1} = h\sum_{j=0}^{M-1}\left|(\Phi_j^n)^\ast\Phi_j^n - \Phi(t_n, x_j)^\ast\Phi(t_n, x_j)\right|,
\ee
the discrete relative $l^1$ error for current density is given by
\begin{equation}\label{eq:err_current}
e_\mathbf{J}^{\eps, \tau}(t_n) = \frac{\|\mathbf{J}(\Phi^n) - \mathbf{J}(\Phi(t_n, \cdot))\|_{l^1}}{\|\mathbf{J}(\Phi(t_n, \cdot))\|_{l^1}},
\end{equation}
where $\mathbf{J}(\Phi^n) = (J_1(\Phi^n), J_2(\Phi^n))^T$, with
\begin{equation}
J_k(\Phi^n) = \frac{1}{\eps}(\Phi^n)^\ast\sigma_k\Phi^n, \quad k = 1, 2,
\end{equation}
and the relative error for energy is defined as
\begin{equation}\label{eq:err_energy}
e_E^{\eps, \tau}(t_n) = \frac{|E(\Phi^n) - E(\Phi(t_n, \cdot))|}{E(\Phi(t_n, \cdot))},
\end{equation}
where
\begin{equation*}
E(\Phi^n) = h\sum_{j=0}^{M-1}\left(-\frac{i}{\eps}(\Phi_j^n)^\ast\sigma_1(\Phi')_j^n + \frac{1}{\eps^2}(\Phi_j^n)^\ast\sigma_3\Phi_j^n + V(x_j)|\Phi_j^n|^2 + \frac{\lambda_1}{2}((\Phi_j^n)^\ast\sigma_3\Phi_j^n)^2 + \frac{\lambda_2}{2}|\Phi_j^n|^4\right).
\end{equation*}

Tables \ref{table:NLDirac_S2} to \ref{table:NLDirac_S2_energy} exhibit the corresponding numerical temporal errors $e^{\eps, \tau}(t = 2\pi)$, $e_\rho^{\eps, \tau}(t=2\pi)$, $e_\mathbf{J}^{\eps, \tau}(t=2\pi)$, and $e_E^{\eps, \tau}(t = 2\pi)$ for $S_2$ with different $\varepsilon$ and resonant time step size $\tau$.

\begin{table}[htp]
	\def\temptablewidth{1\textwidth}
	\vspace{-12pt}
	\caption{Discrete $H^1$ temporal errors $e^{\eps, \tau}(t = 2\pi)$ for the wave function of the  NLDE \eqref{eq:NLDirac_nonmag} with resonant time step size, $S_2$ method. }
	{\rule{\temptablewidth}{1pt}}
	
	\begin{tabular*}{\temptablewidth}{@{\extracolsep{\fill}}cccccccc}
		$e^{\eps, \tau}(t = 2\pi)$	 & $\tau_0 = \pi / 4$ & $\tau_0 / 4$ & $\tau_0 / 4^2$ & $\tau_0 / 4^3$
		& $\tau_0 / 4^4$ & $\tau_0 / 4^5$ & $\tau_0/4^6$ \\ \hline
		$\eps_0 = 1$ & 1.17E+1 & \textbf{2.55E-1} &	1.37E-2 & 8.49E-4 &	5.30E-5 & 3.31E-6 &	2.07E-7\\
		order & -- & \textbf{2.76} &	2.11 &	2.01 &	2.00 &	2.00 &	2.00\\\hline
		$\eps_0 / 2$ & 4.63 & 4.32E-1 &	\textbf{7.83E-3} &	4.84E-4 &	3.02E-5 &	1.89E-6 &	1.18E-7\\
		order & -- & 1.71 &	\textbf{2.89} &	2.01 &	2.00 &	2.00 &	2.00\\\hline
		$\eps_0 / 2^2$ & 4.36 &	1.50 &	1.04E-2 &	\textbf{6.00E-4} &	3.73E-5 &	2.33E-6 &	1.45E-7\\
		order & -- & 0.77 &	3.59 &	\textbf{2.05} &	2.00 &	2.00 &	2.00\\\hline
		$\eps_0 / 2^3$ & 3.61 &	8.39E-1 &	7.79E-1 & 1.02E-3 &	\textbf{5.98E-5} & 3.72E-6 &	2.32E-7\\
		order & -- & 1.05 &	0.05 &	4.79 &	\textbf{2.05} &	2.00 &	2.00\\\hline
		$\eps_0 / 2^4$ & 3.51 &	4.38E-1 & 4.14E-1 &	4.02E-1 & 1.19E-4 &	\textbf{6.95E-6} & 4.32E-7\\
		order & -- & 1.50 &	0.04 &	0.02 &	5.86 &	\textbf{2.05} &	2.00\\\hline
		$\eps_0 / 2^5$ &\textbf{3.50} &	2.44E-1 & 2.09E-1 &	2.08E-1 & 2.05E-1 &	1.47E-5 & \textbf{8.55E-7}\\
		order & \textbf{--} & 1.92 &	0.11 &	0.00 &	0.01 &	6.89 &	\textbf{2.05}\\\hline
		$\eps_0 / 2^9$ & 3.46 &	\textbf{1.10E-1} & 1.45E-2 &	1.31E-2 & 1.31E-2 &	1.31E-2 & 1.31E-2\\
		order & -- & \textbf{2.49} &	1.46 &	0.07 &	0.00 &	0.00 &	0.00\\\hline
		$\eps_0 / 2^{13}$ & 3.45 &	1.08E-1 & \textbf{4.76E-3} &	9.11E-4 & 8.21E-4 &	8.18E-4 & 8.18E-4\\
		order & -- & 2.50 &	\textbf{2.25} &	1.19 &	0.08 &	0.00 &	0.00\\\hline
		$\eps_0 / 2^{17}$ & 3.45 &	1.08E-1 &	4.57E-3 &	\textbf{3.18E-4} &	7.94E-5 &	7.57E-5 &	7.57E-5\\
		order & -- &  2.50 &	2.28 &	\textbf{1.92} &	1.00 &	0.03 &	0.00\\\hline\hline
		$\max\limits_{0<\eps\leq 1}e^{\eps, \tau}(t = 2\pi)$& 1.17E+1 &	1.50 &	7.79E-1 &	4.02E-1 &	2.05E-1 &	1.04E-1 & 5.21E-2\\
		order & -- & 1.48 &	0.47 &	0.48 &	0.49 &	0.49 &	0.50
	\end{tabular*}
	{\rule{\temptablewidth}{1pt}}
	\label{table:NLDirac_S2}
\end{table}

\begin{table}[htp]
	\def\temptablewidth{1\textwidth}
	\vspace{-12pt}
	\caption{Discrete $L^1$ temporal errors $e_\rho^{\eps, \tau}(t = 2\pi)$ for the probability density of the  NLDE \eqref{eq:NLDirac_nonmag} with resonant time step size, $S_2$ method. }
	{\rule{\temptablewidth}{1pt}}
	
	\begin{tabular*}{\temptablewidth}{@{\extracolsep{\fill}}cccccccc}
		$e_\rho^{\eps, \tau}(t = 2\pi)$	 & $\tau_0 = \pi / 4$ & $\tau_0 / 4$ & $\tau_0 / 4^2$ & $\tau_0 / 4^3$
		& $\tau_0 / 4^4$ & $\tau_0 / 4^5$ & $\tau_0/4^6$ \\ \hline
		$\eps_0 = 1$ & 1.79 &	\textbf{3.63E-2} &	2.04E-3 &	1.27E-4 &	7.94E-6 &	4.96E-7 &	3.11E-8\\
		order & -- & \textbf{2.81} &	2.08 &	2.00 &	2.00 &	2.00 &	2.00\\\hline
		$\eps_0 / 2$ & 1.09 &	4.94E-2 &	\textbf{1.56E-3} &	9.66E-5 &	6.03E-6 &	3.77E-7 &	2.37E-8\\
		order & -- & 2.23 &	\textbf{2.49} &	2.01 &	2.00 &	2.00 &	2.00\\\hline
		$\eps_0 / 2^2$ & 1.37 &	4.68E-1 &	2.86E-3 &	\textbf{1.61E-4} &	9.97E-6 &	6.23E-7 &	3.87E-8\\
		order & -- & 0.77 &	3.68 &	\textbf{2.08} &	2.00 &	2.00 &	2.01\\\hline
		$\eps_0 / 2^3$ & 1.06 &	3.87E-1 &	2.97E-1 &	3.05E-4 &	\textbf{1.74E-5} &	1.08E-6 &	6.72E-8\\
		order & -- & 0.73 &	0.19 &	4.96 &	\textbf{2.07} &	2.00 &	2.00\\\hline
		$\eps_0 / 2^4$ & 9.00E-1 &	2.05E-1 &	1.89E-1 &	1.70E-1 &	3.50E-5 &	\textbf{2.01E-6} &	1.25E-7\\
		order & -- & 1.07 &	0.06 &	0.08 &	6.12 &	\textbf{2.06} &	2.00\\\hline
		$\eps_0 / 2^5$ & \textbf{8.28E-1} &	1.13E-1 &	9.58E-2 &	9.49E-2 &	9.02E-2 &	4.20E-6 &	\textbf{2.43E-7}\\
		order & \textbf{--} & 1.44 &	0.12 &	0.01 &	0.04 &	7.20 &	\textbf{2.06}\\\hline
		$\eps_0 / 2^9$ & 7.66E-1 &	\textbf{3.01E-2} &	7.02E-3 &	6.02E-3 &	5.97E-3 &	5.96E-3 &	5.96E-3\\
		order & -- & \textbf{2.33} &	1.05 &	0.11 &	0.01 &	0.00 &	0.00\\\hline
		$\eps_0 / 2^{13}$ & 7.63E-1 &	2.67E-2 &	\textbf{1.84E-3} &	4.39E-4 &	3.76E-4 &	3.73E-4 &	3.73E-4\\
		order & -- & 2.42 &	\textbf{1.93} &	1.03 &	0.11 &	0.01 &	0.00\\\hline
		$\eps_0 / 2^{17}$ & 7.62E-1 &	2.65E-2 &	1.62E-3 &	\textbf{1.14E-4} &	2.60E-5 &	2.27E-5 &	2.27E-5\\
		order & -- &  2.42 &	2.02 &	\textbf{1.92} &	1.06 &	0.10 &	0.00\\\hline\hline
		$\max\limits_{0<\eps\leq 1}e_\rho^{\eps, \tau}(t = 2\pi)$& 1.79 &	4.68E-1 &	2.97E-1 &	1.70E-1 &	9.02E-2 &	4.64E-2 &	2.35E-2\\
		order & -- & 0.97 &	0.33 &	0.40 &	0.46 &	0.48 &	0.49
	\end{tabular*}
	{\rule{\temptablewidth}{1pt}}
	\label{table:NLDirac_S2_density}
\end{table}

\begin{table}[htp]
	\def\temptablewidth{1\textwidth}
	\vspace{-12pt}
	\caption{Discrete relative $L^1$ temporal errors $e_\mathbf{J}^{\eps, \tau}(t = 2\pi)$ for the current density of the  NLDE \eqref{eq:NLDirac_nonmag} with resonant time step size, $S_2$ method. }
	{\rule{\temptablewidth}{1pt}}
	
	\begin{tabular*}{\temptablewidth}{@{\extracolsep{\fill}}cccccccc}
	$e_\mathbf{J}^{\eps, \tau}(t = 2\pi)$	 & $\tau_0 = \pi / 4$ & $\tau_0 / 4$ & $\tau_0 / 4^2$ & $\tau_0 / 4^3$
	& $\tau_0 / 4^4$ & $\tau_0 / 4^5$ & $\tau_0/4^6$ \\ \hline
	$\eps_0 = 1$ & 7.11E-1 &	\textbf{1.47E-2} &	8.30E-4 &	5.16E-5 &	3.22E-6 &	2.02E-7 &	1.26E-8\\
	order & -- & \textbf{2.80} &	2.07 &	2.00 &	2.00 &	2.00 &	2.00\\\hline
	$\eps_0 / 2$ & 5.93E-1 &	2.55E-2 &	\textbf{8.37E-4} &	5.18E-5 &	3.23E-6 &	2.02E-7 &	1.27E-8\\
	order & -- & 2.27 &	\textbf{2.46} &	2.01 &	2.00 &	2.00 &	2.00\\\hline
	$\eps_0 / 2^2$ & 5.71E-1 &	3.34E-1 &	1.74E-3 &	\textbf{9.99E-5} &	6.22E-6 &	3.88E-7 &	2.41E-8\\
	order & -- & 0.39 &	3.79 &	\textbf{2.06} &	2.00 &	2.00 &	2.00\\\hline
	$\eps_0 / 2^3$ & 4.14E-1 &	2.19E-1 &	2.06E-1 &	1.98E-4 &	\textbf{1.15E-5} &	7.18E-7 &	4.47E-8\\
	order & -- & 0.46 &	0.05 &	5.01 &	\textbf{2.05} &	2.00 &	2.00\\\hline
	$\eps_0 / 2^4$ & 3.58E-1 &	1.17E-1 &	1.16E-1 &	1.13E-1 &	2.36E-5 &	\textbf{1.38E-6} &	8.56E-8\\
	order & -- & 0.81 &	0.01 &	0.02 &	6.11 &	\textbf{2.05} &	2.00\\\hline
	$\eps_0 / 2^5$ &\textbf{3.46E-1} &	6.07E-2 &	5.95E-2 &	5.95E-2 &	5.88E-2 &	2.90E-6 &	\textbf{1.69E-7}\\
	order & \textbf{--} & 1.26 &	0.01 &	0.00 &	0.01 &	7.16 &	\textbf{2.05}\\\hline
	$\eps_0 / 2^9$ & 3.42E-1 &	\textbf{1.28E-2} &	3.85E-3 &	3.81E-3 &	3.81E-3 &	3.81E-3 &	3.81E-3\\
	order & -- & \textbf{2.37} &	0.86 &	0.01 &	0.00 &	0.00 &	0.00\\\hline
	$\eps_0 / 2^{13}$ & 3.42E-1 &	1.24E-2 &	\textbf{7.76E-4} &	2.41E-4 &	2.38E-4 &	2.38E-4 &	2.38E-4\\
	order & -- & 2.39 &	\textbf{2.00} &	0.84 &	0.01 &	0.00 &	0.00\\\hline
	$\eps_0 / 2^{17}$ & 3.42E-1 &	1.24E-2 &	7.51E-4 &	\textbf{4.68E-5} &	1.35E-5 &	1.37E-5 &	1.37E-5\\
	order & -- &  2.39 &	2.02 &	\textbf{2.00} &	0.90 &	-0.01 &	0.00\\\hline\hline
	$\max\limits_{0<\eps\leq 1}e_\mathbf{J}^{\eps, \tau}(t = 2\pi)$& 7.11E-1 &	3.34E-1	 & 2.06E-1 &	1.13E-1 &	5.88E-2 &	3.00E-2 &	1.51E-2\\
	order & -- & 0.55 &	0.35 &	0.43 &	0.47 &	0.49 &	0.49
\end{tabular*}
{\rule{\temptablewidth}{1pt}}
\label{table:NLDirac_S2_current}
\end{table}

\begin{table}[htp]
	\def\temptablewidth{1\textwidth}
	\vspace{-12pt}
	\caption{Relative temporal errors $e_E^{\eps, \tau}(t = 2\pi)$ for the energy of the  NLDE \eqref{eq:NLDirac_nonmag} with resonant time step size, $S_2$ method. }
	{\rule{\temptablewidth}{1pt}}
	
	\begin{tabular*}{\temptablewidth}{@{\extracolsep{\fill}}cccccccc}
		$e_E^{\eps, \tau}(t = 2\pi)$	 & $\tau_0 = \pi / 4$ & $\tau_0 / 4$ & $\tau_0 / 4^2$ & $\tau_0 / 4^3$
		& $\tau_0 / 4^4$ & $\tau_0 / 4^5$ & $\tau_0/4^6$ \\ \hline
		$\eps_0 = 1$ & 1.30E-1 &	\textbf{1.94E-3} &	1.10E-4 &	6.86E-6 &	4.29E-7 &	2.69E-8 &	1.78E-9\\
		order & -- & \textbf{3.03} &	2.07 &	2.00 &	2.00 &	2.00 &	1.96\\\hline
		$\eps_0 / 2$ & 3.29E-2 &	2.16E-3 &	\textbf{7.02E-5} &	4.27E-6 &	2.67E-7 &	1.67E-8 &	1.12E-9 \\
		order & -- & 1.96 &	\textbf{2.47} &	2.02 &	2.00 &	2.00 &	1.95\\\hline
		$\eps_0 / 2^2$ & 2.01E-2 &	2.53E-2 &	2.54E-4 &	\textbf{1.36E-5} &	8.44E-7 &	5.25E-8 &	3.10E-9\\
		order & -- & -0.17 &	3.32 &	\textbf{2.11} &	2.01 &	2.00 &	2.04\\\hline
		$\eps_0 / 2^3$ & 3.20E-2 &	1.50E-3 &	9.21E-3 &	3.37E-5 &	\textbf{1.90E-6} &	1.18E-7 &	7.07E-9\\
		order & -- & 2.21 &	-1.31 &	4.05 &	\textbf{2.08} &	2.01 &	2.03\\\hline
		$\eps_0 / 2^4$ & 4.05E-2 &	4.25E-4 &	1.65E-3 &	3.32E-3 &	4.29E-6 &	\textbf{2.45E-7} &	1.51E-8\\
		order & -- & 3.29 &	-0.98 &	-0.50 &	4.80 &	\textbf{2.06} &	2.01\\\hline
		$\eps_0 / 2^5$ &\textbf{4.45E-2} &	1.50E-3 &	7.52E-4 &	8.92E-4 &	1.29E-3 &	5.35E-7 &	\textbf{3.09E-8}\\
		order & \textbf{--} & 2.45 &	0.50 &	-0.12 &	-0.27 &	5.62 &	\textbf{2.06}\\\hline
		$\eps_0 / 2^9$ & 4.65E-2 &	\textbf{2.51E-3} &	1.05E-4 &	4.42E-5 &	5.35E-5 &	5.41E-5 &	5.42E-5\\
		order & -- & \textbf{2.11} &	2.29 &	0.62 &	-0.14 &	-0.01 &	0.00\\\hline
		$\eps_0 / 2^{13}$ & 4.66E-2 &	2.57E-3 &	\textbf{1.55E-4} &	6.54E-6 &	2.75E-6 &	3.33E-6 &	3.36E-6\\
		order & -- & 2.09 &	\textbf{2.02} &	2.28 &	0.63 &	-0.14 &	-0.01\\\hline
		$\eps_0 / 2^{17}$ & 4.66E-2 &	2.57E-3 &	1.59E-4 &	\textbf{1.03E-5} &	1.03E-6 &	4.49E-7 &	4.49E-7\\
		order & -- &  2.09 &	2.01 &	\textbf{1.97} &	1.66 &	0.60 &	0.00\\\hline\hline
		$\max\limits_{0<\eps\leq 1}e_E^{\eps, \tau}(t = 2\pi)$& 1.30E-1 &	2.53E-2 &	9.21E-3 &	3.32E-3 &	1.29E-3 &	5.44E-4 &	2.45E-4\\
		order & -- & 1.18 &	0.73 &	0.74 &	0.68 &	0.62 &	0.58
	\end{tabular*}
	{\rule{\temptablewidth}{1pt}}
	\label{table:NLDirac_S2_energy}
\end{table}

In these tables, the last two rows show the largest error of each column for fixed $\tau$. We could observe similar patterns for the errors of the wave function, and the physical observables.
Clearly, overall there is $1/2$ order convergence, which agrees well with Theorem \ref{thm:strang} for the wave function, and also suggests the same convergence rate for the observables.
More specifically, from Tables \ref{table:NLDirac_S2} to \ref{table:NLDirac_S2_energy}, we can see when $\tau\gtrsim \sqrt{\eps}$ (below the lower bolded diagonal line), there is second order convergence, which coincides with the error bound  $\tau^2 +\eps$; when $\tau\lesssim\eps^2$ (above the upper bolded diagonal line), we  also observe second order convergence, which matches the other error bound $\tau^2 +\tau^2/ \eps^3$.\\

Furthermore, to support the improved uniform error bound in Theorems \ref{thm:strang2}, we test the error bounds using non-resonant time step sizes, i.e., we choose $\tau\in\mathcal{A}_\delta(\eps)$ for some given $\eps$ and fixed $0<\delta \leq1$. Similar to the resonant time step case, we also test the errors for physical observables.
The bounded computational domain is set as $\Omega = (-16, 16)$.

For comparison, the numerical `exact' solution is computed by $S_2$ with a very small time step size $\tau_e = 8\times10^{-6}$. Spatial mesh size is fixed as $h=1/16$ for all the numerical simulations.

Tables \ref{table:NLDirac_S2_nrs} to \ref{table:NLDirac_S2_nrs_energy} show the numerical temporal errors $e^{\eps, \tau}(t = 4)$, $e_\rho^{\eps, \tau}(t=4)$, $e_\mathbf{J}^{\eps, \tau}(t=4)$, and $e_E^{\eps, \tau}(t = 4)$ with different $\varepsilon$ and non-resonant time step size $\tau$ for $S_2$.

\begin{table}[htp]
	\def\temptablewidth{1\textwidth}
	\vspace{-12pt}
	\caption{Discrete $H^1$ temporal errors $e^{\eps, \tau}(t = 4)$ for the wave function with non-resonant time step size, $S_2$ method. }
	{\rule{\temptablewidth}{1pt}}
	\begin{tabular*}{\temptablewidth}{@{\extracolsep{\fill}}ccccccc}
		$e^{\eps, \tau}(t = 4)$& $\tau_0 = 1 / 4$ & $\tau_0 / 4$ & $\tau_0 / 4^2$ & $\tau_0 / 4^3$ & $\tau_0 / 4^4$ & $\tau_0 / 4^5$\\ \hline
		$\eps_0 = 1$ & 3.34E-1 &	\textbf{1.74E-2} &	1.08E-3 &	6.74E-5 &	4.21E-6 &	2.63E-7\\
		order & -- & \textbf{2.13} &	2.01 &	2.00 &	2.00 &	2.00\\\hline
		$\eps_0 / 2$ &1.53 &	9.43E-3 &	\textbf{5.83E-4} &	3.64E-5 &	2.27E-6 &	1.42E-7\\
		order & -- & 3.67 &	\textbf{2.01} &	2.00 &	2.00 &	2.00\\\hline
		$\eps_0 / 2^2$ & 6.44E-1 &	1.70E-2 &	8.76E-4 &	\textbf{5.44E-5} &	3.40E-6 &	2.12E-7\\
		order & -- & 2.62 &	2.14 &	\textbf{2.00} &	2.00 &	2.00\\\hline
		$\eps_0 / 2^3$ & 5.41E-1 &	5.31E-2 &	1.83E-3 &	9.64E-5 &	\textbf{5.98E-6} &	3.73E-7\\
		order & -- & 1.67 &	2.43 &	2.12 &	\textbf{2.01} &	2.00\\\hline
		$\eps_0 / 2^4$ & 2.67E-1 &	1.09E-1 &	6.97E-3 &	2.06E-4 &	1.15E-5 &	\textbf{7.15E-7}\\
		order & -- & 0.65 &	1.98 &	2.54 &	2.08 &	\textbf{2.01}\\\hline
		$\eps_0 / 2^6$ &\textbf{2.55E-1} &	9.94E-3 &	2.12E-3 &	9.51E-4 &	1.09E-4 &	3.21E-6\\
		order & \textbf{--} & 2.34 &	1.11 &	0.58 &	1.56 &	2.54\\\hline
		$\eps_0 / 2^8$ & 2.11E-1 &	\textbf{1.05E-2} &	6.37E-3 &	1.01E-4 &	1.83E-5 &	1.52E-5\\
		order & -- & \textbf{2.17} &	0.36 &	2.99 &	1.23 &	0.13\\\hline
		$\eps_0 / 2^{10}$ & 2.09E-1 &	8.41E-3 &	\textbf{2.09E-3} &	5.13E-5	 & 2.95E-5 &	1.14E-6\\
		order & -- & 2.32 &	\textbf{1.00} &	2.67 &	0.40 &	2.35\\\hline
		$\eps_0 / 2^{12}$ & 2.10E-1 &	8.43E-3 &	2.16E-3 &	\textbf{3.87E-5} &	2.86E-6 &	4.81E-7\\
		order & -- & 2.32 &	0.98 &	\textbf{2.90} &	1.88 &	1.29\\\hline
		$\eps_0 / 2^{14}$ & 2.10E-1 &	8.42E-3 &	2.14E-3 &	3.84E-5 &	\textbf{3.71E-6} &	4.82E-7\\
		order & -- & 2.32 &	0.99 &	2.90 &	\textbf{1.69} &	1.47\\\hline\hline
		$\max\limits_{0<\eps\leq 1}e^{\eps, \tau}(t = 4)$ & 1.53 &	1.09E-1 &	7.36E-3 &	9.51E-4 &	1.21E-4 &	1.52E-5\\
		order & -- & 1.91 &	1.94 &	1.48 &	1.49 &	1.49
	\end{tabular*}
	{\rule{\temptablewidth}{1pt}}
	\label{table:NLDirac_S2_nrs}
\end{table}

\begin{table}[htp]
	\def\temptablewidth{1\textwidth}
	\vspace{-12pt}
	\caption{Discrete $L^1$ temporal errors $e_\rho^{\eps, \tau}(t = 4)$ for the probability density with non-resonant time step size, $S_2$ method. }
	{\rule{\temptablewidth}{1pt}}
	\begin{tabular*}{\temptablewidth}{@{\extracolsep{\fill}}ccccccc}
		$e_\rho^{\eps, \tau}(t = 4)$& $\tau_0 = 1 / 4$ & $\tau_0 / 4$ & $\tau_0 / 4^2$ & $\tau_0 / 4^3$ & $\tau_0 / 4^4$ & $\tau_0 / 4^5$\\ \hline
		$\eps_0 = 1$ & 3.59E-2 &	\textbf{1.98E-3} &	1.23E-4 &	7.69E-6 &	4.81E-7 &	3.01E-8 \\
		order & -- & \textbf{2.09} &	2.00 &	2.00 &	2.00 &	2.00\\\hline
		$\eps_0 / 2$ &1.79E-1 &	1.94E-3 &	\textbf{1.19E-4} &	7.45E-6 &	4.66E-7 &	2.90E-8\\
		order & -- & 3.27 &	\textbf{2.01} &	2.00 &	2.00 &	2.00\\\hline
		$\eps_0 / 2^2$ & 1.18E-1 &	3.90E-3 &	2.14E-4 &	\textbf{1.32E-5} &	8.27E-7 &	5.16E-8\\
		order & -- & 2.46 &	2.09 &	\textbf{2.01} &	2.00 &	2.00\\\hline
		$\eps_0 / 2^3$ & 1.85E-1 &	1.42E-2 &	4.46E-4 &	2.38E-5 &	\textbf{1.47E-6}  &	9.20E-8\\
		order & -- & 1.85 &	2.50 &	2.11 &	\textbf{2.01} &	2.00\\\hline
		$\eps_0 / 2^4$ & 3.35E-2 &	1.34E-2 &	1.30E-3 &	4.24E-5 &	2.33E-6 &	\textbf{1.45E-7}\\
		order & -- & 0.66 &	1.68 &	2.47 &	2.09 &	\textbf{2.01}\\\hline
		$\eps_0 / 2^6$ & \textbf{4.75E-2} &	2.45E-3 &	2.27E-4 &	2.11E-4 &	2.31E-5 &	7.01E-7\\
		order & \textbf{--} & 2.14 &	1.72 &	0.05 &	1.60 &	2.52\\\hline
		$\eps_0 / 2^8$ & 3.36E-2 &	\textbf{2.01E-3} &	4.43E-4 &	1.59E-5 &	3.36E-6 &	2.87E-6\\
		order & -- & \textbf{2.03} &	1.09 &	2.40 &	1.12 &	0.11\\\hline
		$\eps_0 / 2^{10}$ & 3.30E-2 &	1.93E-3 &	\textbf{1.35E-4} &	1.17E-5 &	6.61E-6 &	2.77E-7\\
		order & -- & 2.05 &	\textbf{1.92} &	1.76 &	0.41 &	2.29\\\hline
		$\eps_0 / 2^{12}$ & 3.35E-2 &	1.95E-3 &	1.37E-4 &	\textbf{7.50E-6} &	6.66E-7 &	9.86E-8\\
		order & -- & 2.05 &	1.92 &	\textbf{2.09} &	1.75 &	1.38\\\hline
		$\eps_0 / 2^{14}$ & 3.35E-2 &	1.96E-3 &	1.24E-4 &	7.52E-6 &	\textbf{8.93E-7} &	1.61E-7\\
		order & -- & 2.05 &	1.99 &	2.02 &	\textbf{1.54} &	1.24\\\hline\hline
		$\max\limits_{0<\eps\leq 1}e_\rho^{\eps, \tau}(t = 4)$ & 1.85E-1 &	1.43E-2 &	1.57E-3 &	2.11E-4 &	2.37E-5 &	2.87E-6\\
		order & -- & 1.85 &	1.59 &	1.45 &	1.58 &	1.52
	\end{tabular*}
	{\rule{\temptablewidth}{1pt}}
	\label{table:NLDirac_S2_nrs_density}
\end{table}

\begin{table}[htp]
	\def\temptablewidth{1\textwidth}
	\vspace{-12pt}
	\caption{Discrete relative $L^1$ temporal errors $e_\mathbf{J}^{\eps, \tau}(t = 4)$ for the current density with non-resonant time step size, $S_2$ method. }
	{\rule{\temptablewidth}{1pt}}
	\begin{tabular*}{\temptablewidth}{@{\extracolsep{\fill}}ccccccc}
		$e_\mathbf{J}^{\eps, \tau}(t = 4)$ & $\tau_0 = 1 / 4$ & $\tau_0 / 4$ & $\tau_0 / 4^2$ & $\tau_0 / 4^3$ & $\tau_0 / 4^4$ & $\tau_0 / 4^5$\\ \hline
		$\eps_0 = 1$ & 2.07E-2 &	\textbf{1.13E-3} &	7.05E-5 &	4.40E-6 &	2.75E-7 &	1.72E-8\\
		order & -- & \textbf{2.10} &	2.00 &	2.00 &	2.00 &	2.00\\\hline
		$\eps_0 / 2$ & 8.37E-2 &	1.07E-3 &	\textbf{6.60E-5} &	4.12E-6 &	2.58E-7 &	1.61E-8\\
		order & -- & 3.15 &	\textbf{2.01} &	2.00 &	2.00 &	2.00\\\hline
		$\eps_0 / 2^2$ & 7.70E-2 &	2.52E-3 &	1.36E-4 &	\textbf{8.46E-6} &	5.28E-7 &	3.30E-8\\
		order & -- & 2.47 &	2.10 &	\textbf{2.00} &	2.00 &	2.00\\\hline
		$\eps_0 / 2^3$ & 1.06E-1 &	9.81E-3 &	2.87E-4 &	1.60E-5 &	\textbf{9.96E-7} &	6.21E-8\\
		order & -- & 1.71 &	2.55 &	2.08 &	\textbf{2.00} &	2.00\\\hline
		$\eps_0 / 2^4$ & 1.73E-2 &	9.97E-3 &	1.20E-3 &	3.41E-5 &	1.92E-6 &	\textbf{1.19E-7}\\
		order & -- & 0.40 &	1.53 &	2.57 &	2.08 &	\textbf{2.01}\\\hline
		$\eps_0 / 2^6$ & \textbf{4.76E-2} &	1.63E-3 &	1.89E-4 &	1.71E-4 &	1.91E-5 &	5.62E-7\\
		order & \textbf{--} & 2.43 &	1.56 &	0.07 &	1.58 &	2.54\\\hline
		$\eps_0 / 2^8$ & 1.97E-2 &	\textbf{1.28E-3} &	3.92E-4 &	1.38E-5 &	3.04E-6 &	2.59E-6\\
		order & -- & \textbf{1.97} &	0.85 &	2.42 &	1.09 &	0.12\\\hline
		$\eps_0 / 2^{10}$ & 2.02E-2 &	1.10E-3 &	\textbf{8.13E-5} &	7.89E-6 &	4.86E-6 &	2.02E-7\\
		order & -- & 2.10 &	\textbf{1.88} &	1.68 &	0.35 &	2.30\\\hline
		$\eps_0 / 2^{12}$ & 1.91E-2 &	1.11E-3 &	8.95E-5 &	\textbf{4.05E-6} &	4.88E-7 &	8.15E-8\\
		order & -- & 2.05 &	1.81 &	\textbf{2.23} &	1.53 &	1.29\\\hline
		$\eps_0 / 2^{14}$ & 1.91E-2 &	1.12E-3 &	7.03E-5 &	4.18E-6 &	\textbf{6.63E-7} &	1.31E-7\\
		order & -- & 2.05 &	2.00 &	2.04 &	\textbf{1.33} &	1.17\\\hline\hline
		$\max\limits_{0<\eps\leq 1}e_\mathbf{J}^{\eps, \tau}(t = 4)$ & 1.06E-1 &	9.97E-3 &	1.27E-3 &	1.71E-4 &	2.01E-5 &	2.59E-6\\
		order & -- & 1.70 &	1.49 &	1.45 &	1.54 &	1.48
	\end{tabular*}
	{\rule{\temptablewidth}{1pt}}
	\label{table:NLDirac_S2_nrs_current}
\end{table}

\begin{table}[htp]
	\def\temptablewidth{1\textwidth}
	\vspace{-12pt}
	\caption{Relative temporal errors $e_E^{\eps, \tau}(t = 4)$ for the energy with non-resonant time step size, $S_2$ method. }
	{\rule{\temptablewidth}{1pt}}
	\begin{tabular*}{\temptablewidth}{@{\extracolsep{\fill}}ccccccc}
		$e_E^{\eps, \tau}(t = 4)$& $\tau_0 = 1 / 4$ & $\tau_0 / 4$ & $\tau_0 / 4^2$ & $\tau_0 / 4^3$ & $\tau_0 / 4^4$ & $\tau_0 / 4^5$\\ \hline
		$\eps_0 = 1$ & 5.28E-4 & \textbf{1.29E-5} &	7.85E-7 &	4.89E-8 &	3.00E-9 &	1.33E-10\\
		order & -- & \textbf{2.67} &	2.02 &	2.00 &	2.01 &	2.25\\\hline
		$\eps_0 / 2$ & 1.75E-2 &	1.43E-4 &	\textbf{8.64E-6} &	5.39E-7 &	3.36E-8 &	2.04E-9\\
		order & -- & 3.47 &	\textbf{2.02} &	2.00 &	2.00 &	2.02\\\hline
		$\eps_0 / 2^2$ & 1.57E-2 &	3.73E-4 &	1.84E-5 &	\textbf{1.14E-6} &	7.10E-8 &	4.40E-9\\
		order & -- & 2.70 &	2.17 &	\textbf{2.01} &	2.00 &	2.01\\\hline
		$\eps_0 / 2^3$ & 3.41E-2 &	2.14E-3 &	5.49E-5 &	2.97E-6 &	\textbf{1.84E-7} &	1.13E-8\\
		order & -- & 2.00 &	2.64 &	2.10 &	\textbf{2.01} &	2.01\\\hline
		$\eps_0 / 2^4$ & 2.57E-3 &	3.03E-3 &	2.91E-4 &	8.13E-6 &	4.48E-7 &	\textbf{2.78E-8}\\
		order & -- & -0.12 &	1.69 &	2.58 &	2.09 &	\textbf{2.01}\\\hline
		$\eps_0 / 2^6$ & \textbf{1.22E-2} &	2.98E-4 &	3.86E-5 &	3.68E-5 &	4.05E-6 &	1.19E-7\\
		order & \textbf{--} & 2.68 &	1.47 &	0.03 &	1.59 &	2.55\\\hline
		$\eps_0 / 2^8$ & 1.74E-3 &	\textbf{1.79E-4} &	8.27E-5 &	3.16E-6 &	7.20E-7 &	6.16E-7\\
		order & -- & \textbf{1.64} &	0.56 &	2.35 &	1.07 &	0.11\\\hline
		$\eps_0 / 2^{10}$ & 1.98E-3 &	7.99E-5 &	\textbf{1.11E-5} &	1.53E-6 &	1.10E-6 &	4.60E-8\\
		order & -- & 2.31 &	\textbf{1.42} &	1.43 &	0.24 &	2.29\\\hline
		$\eps_0 / 2^{12}$ & 1.35E-3 &	8.45E-5 &	1.64E-5 &	\textbf{1.09E-7} &	1.11E-7 &	2.24E-8\\
		order & -- & 2.00 &	1.18 &	\textbf{3.62} &	-0.01 &	1.15\\\hline
		$\eps_0 / 2^{14}$ & 1.37E-3 &	9.61E-5 &	5.86E-6 &	2.34E-7 &	\textbf{1.81E-7} &	8.06E-9\\
		order & -- & 1.92 &	2.02 &	2.32 &	\textbf{0.19} &	2.24\\\hline\hline
		$\max\limits_{0<\eps\leq 1}e_E^{\eps, \tau}(t = 4)$ & 3.41E-2 &	3.03E-3 &	3.69E-4 &	3.95E-5 &	5.66E-6 &	6.28E-7\\
		order & -- & 1.75 &	1.52 &	1.61 &	1.40 &	1.59
	\end{tabular*}
	{\rule{\temptablewidth}{1pt}}
	\label{table:NLDirac_S2_nrs_energy}
\end{table}

The last two rows in Table \ref{table:NLDirac_S2_nrs} to \ref{table:NLDirac_S2_nrs_energy} show the largest error of each column for fixed $\tau$, which gives $3/2$ order of uniform convergence, and it is consistent with Theorem \ref{thm:strang2} for the wave function. We could conclude that for physical observables, the convergence rate is the same.
More specifically, in these tables, we can roughly observe the second order convergence when $\tau\gtrsim\eps$ (below the lower bolded diagonal line) or when $\tau\lesssim\eps^2$ (above the upper bolded diagonal line), agreeing with the error bound $\tau^2+\tau\eps$ and the other error bound $\tau^2+\tau^2/\eps$, respectively. When $\tau$ is large, the performance of the algorithm for probability density and current density is better than the performance for wave function and energy.

Through the results of this example, we successfully validate the uniform error bounds of $S_2$ in Theorems \ref{thm:strang} and \ref{thm:strang2}.

\begin{rmk}
	Through extensive numerical results not shown here for brevity, we found out that the super-resolution property also holds true for higher order time-splitting methods in solving the NLDE. Specifically, the fourth-order compact splitting method for the Dirac equation \cite{BY} and
the fourth-order partitioned Runge-Kutta splitting method for the NLDE \cite{BM,BCJY} exhibits 1/2 order uniform convergence under resonant time steps, and the uniform order could be improved to 3/2 under non-resonant time steps. The details are omitted here for brevity.
\end{rmk}

\section{Conclusion}\setcounter{equation}{0}
We studied the super-resolution property of time-splitting methods for the nonlinear Dirac equation in the nonrelativistic regime without magnetic potential in this paper. The uniform and improved uniform error bounds under non-resonant time step sizes for Lie-Trotter splitting ($S_1$) and Strang splitting ($S_2$) were rigorously established.  For $S_1$,  there are two independent error bounds $\tau+\eps$ and $\tau + \tau/\eps$, which give a uniform $1/2$ order convergence. Surprisingly, there is  an improved  uniform first order  convergence if the time step sizes are non-resonant. For $S_2$, the two different error bounds are $\tau^2 + \eps$ and $\tau^2 + \tau^2/\eps^3$, also resulting in a uniform $1/2$ order convergence. For non-resonant time step sizes, the convergence rates can be improved to $3/2$ for $S_2$, with the two independent error bounds  as $\tau^2 + \tau\eps$ and $\tau^2 + \tau^2/\eps$. Numerical results  agreed with our theorems and suggested that our estimates are sharp. We remark that super-resolution also holds true for higher order splitting methods. Moreover, although only 1D cases are presented in this paper, these results are valid in higher dimensions, and the proofs can be easily generalized.

\end{document}